\documentclass[11pt]{article}
\usepackage[spanish,english]{babel}
\usepackage[utf8]{inputenc}
\usepackage[width=15.5cm,height=21cm]{geometry}

\usepackage{titling}
\usepackage{tikz}
\usepackage{amsmath,amsfonts,amssymb,mathtools}
\usepackage{comment}
\usepackage{enumitem}

\usepackage{amsthm}
\newtheorem{thm}{Theorem}
\newtheorem{prop}[thm]{Proposition}
\newtheorem{lem}[thm]{Lemma}
\newtheorem{cor}[thm]{Corollary}

\newtheorem{remark}[thm]{Remark}

\usepackage[colorlinks=true,allcolors=blue]{hyperref} 
\newcommand{\doi}[1]{\href{https://doi.org/#1}{doi:#1}}
\newcommand{\myurl}[1]{\href{#1}{#1}}

\newcommand{\pheq}{\phantom{=}\ }
\newcommand{\eqdef}{\coloneqq}
\newcommand{\al}{\alpha}
\newcommand{\om}{\omega}
\newcommand{\ga}{\gamma}
\newcommand{\be}{\beta}
\newcommand{\de}{\delta}
\newcommand{\eps}{\varepsilon}
\newcommand{\si}{\sigma}
\newcommand{\ka}{\varkappa}

\newcommand{\la}{\lambda}
\renewcommand{\phi}{\varphi}
\newcommand{\Tht}{\Theta}
\newcommand{\tht}{\vartheta}
\newcommand{\bC}{\mathbb{C}}

\newcommand{\bR}{\mathbb{R}}

\newcommand{\cK}{\mathcal{K}}

\newcommand{\clos}{\operatorname{cl}}
\renewcommand{\Re}{\operatorname{Re}}

\newcommand{\laasympt}{\lambda^{\operatorname{asympt}}}

\newcommand{\medstrut}{\vphantom{\int_0^1}}
\newcommand{\hstrut}{\mbox{}\ \mbox{}}
\newcommand{\bigstrut}{\vphantom{\int_{0_0}^{1^1}}}

\newcommand{\rmA}{\mathrm{a}}
\newcommand{\rmB}{\mathrm{b}}
\newcommand{\rmC}{\mathrm{c}}
\newcommand{\rmD}{\mathrm{d}}

\definecolor{mygreen}{rgb}{0.0,0.75,0.0}

\title{Eigenvalues
of the laplacian matrices \\ of the cycles
with one weighted edge}
\author{Sergei M. Grudsky,
Egor A. Maximenko and Alejandro Soto-Gonz\'alez}
\date{\today\vspace{-6em}}

\begin{document}

\maketitle

\let\thefootnote\relax
\footnote{Sergei M. Grudsky,
CINVESTAV del IPN,
Departamento de Matem\'aticas,
Apartado Postal 07360,
Ciudad de M\'exico,
Mexico.
\href{mailto:grudsky@math.cinvestav.mx}{grudsky@math.cinvestav.mx},
\myurl{https://orcid.org/0000-0002-3748-5449},\\
\myurl{https://publons.com/researcher/2095797/sergei-m-grudsky}.
\vspace*{0.5em}
}

\footnote{Egor A. Maximenko, Instituto Polit\'ecnico Nacional,
Escuela Superior de F\'isica y Matem\'aticas,
Apartado Postal 07730,
Ciudad de M\'exico,
Mexico.
\href{mailto:emaximenko@ipn.mx}{emaximenko@ipn.mx},
\myurl{https://orcid.org/0000-0002-1497-4338}.
\vspace*{0.5em}
}

\footnote{Alejandro Soto-Gonz\'{a}lez, CINVESTAV  del IPN,
Departamento de Matem\'aticas,
Apartado Postal 07360,
Ciudad de M\'exico,
Mexico.
\href{mailto:asoto@math.cinvestav.mx}{asoto@math.cinvestav.mx},
\myurl{https://orcid.org/0000-0003-2419-4754}.
}

\footnote{

\medskip
\textbf{Funding.}

\medskip
The research of the first author has been supported by CONACYT (Mexico) project ``Ciencia de Frontera'' FORDECYT-PRONACES/61517/2020
and by Regional Mathematical Center of the Southern Federal University with the support of
the Ministry of Science and Higher Education of Russia, Agreement 075-02-2021-1386.

\medskip
The research of the second author has been supported by CONACYT (Mexico) project ``Ciencia de Frontera'' FORDECYT-PRONACES/61517/2020 and IPN-SIP projects
(Instituto Polit\'{e}cnico Nacional, Mexico).

\medskip
The research of the third author has been supported by CONACYT (Mexico) PhD scholarship.
}

\begin{abstract}
In this paper we study the eigenvalues of the laplacian matrices of the cyclic graphs with one edge of weight $\alpha$
and the others of weight $1$.
We denote by $n$ the order of the graph and suppose that
$n$ tends to infinity.
We notice that the characteristic polynomial and the eigenvalues depend only on $\operatorname{Re}(\alpha)$.
After that, through the rest of the paper we suppose that $0<\alpha<1$.
It is easy to see that the eigenvalues belong to $[0,4]$ and are asymptotically distributed as the function $g(x)=4\sin^2(x/2)$ on $[0,\pi]$. 
We obtain a series of results about the individual behavior of the eigenvalues. 
First, we describe more precisely their localization in subintervals of $[0,4]$.
Second, we transform the characteristic equation to a form convenient to solve by numerical methods.
In particular, we prove that Newton's method converges for every $n\ge3$.
Third, we derive asymptotic formulas for all eigenvalues,
where the errors are uniformly bounded with respect to the number of the eigenvalue.

\medskip
\textbf{Keywords}: 
eigenvalue, laplacian matrix, weighted cycle, periodic Jacobi matrix, Toeplitz matrix, tridiagonal matrix,
perturbation, asymptotic expansion.

\medskip
\textbf{Mathematics Subject Classification (2020)}:
05C50, 15B05, 47B36, 15A18, 41A60, 65F15, 82B20.
\end{abstract}

\clearpage

\tableofcontents

\section{Introduction}

For every natural $n\ge 3$ and every real $\al$, we denote by $G_{\al,n}$ the cyclic graph of order $n$,
where the edge between the vertices $1$ and $n$ has weight $\al$, and all other edges have weights $1$.
See Figure~\ref{fig:graph} for $n=7$.
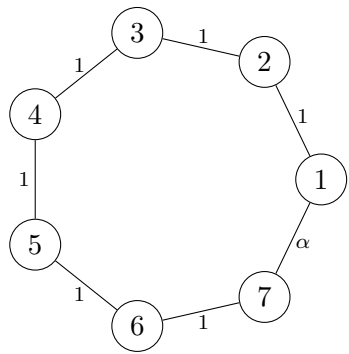
\begin{figure}[th]
\centering
\begin{tikzpicture}
\foreach \j/\k in {0/1,1/2,2/3,3/4,4/5,5/6,6/7} {
  \node (N\j) at (\j*360/7:2cm) [draw, circle] {$\k$};
}
\foreach \j/\k in {0/1,1/2,2/3,3/4,4/5,5/6,6/0} {
  \draw (N\j) -- (N\k);
}
\foreach \j in {0,1,2,3,4,5} {
  \node at (\j*360/7 + 360/14:1.95cm) {$\scriptstyle 1$};
}
\node at (-360/14:1.95cm) {$\scriptstyle\al$};
\end{tikzpicture}
\caption{Graph $G_{\al,7}$\label{fig:graph}}
\end{figure}

\noindent
Let $L_{\al,n}$ be the laplacian matrix of $G_{\al,n}$.
For example,
\begin{equation}\label{eq:lapl_matrix}
L_{\al,7}
=
\begin{bmatrix*}[r]
1+\al & -1 & 0 & 0 & 0 & 0 & -\al\phantom{\al} \\
-1\phantom{\al} & 2 & -1 & 0 & 0 & 0 & 0\phantom{\al} \\
0\phantom{\al} & -1 & 2 & -1 & 0 & 0 & 0\phantom{\al} \\
0\phantom{\al} & 0 & -1 & 2 & -1 & 0 & 0\phantom{\al} \\
0\phantom{\al} & 0 & 0 & -1 & 2 & -1 & 0\phantom{\al} \\
0\phantom{\al} & 0 & 0 & 0 & -1 & 2 & -1\phantom{\al} \\
-\al\phantom{\al} & 0 & 0 & 0 &  0 & -1 & 1+\al
\end{bmatrix*}.
\end{equation}
The spectral decomposition of $L_{\al,n}$ is crucial to solve the heat and wave equations on the graph $G_{\al,n}$, i.e.,
the linear systems of differential equations of the form $f'(t)=-c L_{\al,n} f(t)$ and $f''(t)=c L_{\al,n} f(t)$, where $f(t)=[f_j(t)]_{j=1}^n$
and $c$ is some coefficient.
Moreover, laplacian matrices appear in the study in of random walks on graphs, electrical flows, network dynamics, and many other physical phenomena; see, e.g.~\cite{M2012}.

The matrices $L_{\al,n}$ can also be viewed as periodic Jacobi matrices
and as real symmetric Toeplitz matrices with perturbations on the corners $(1,1)$, $(1,n)$, $(n,1)$, and $(n,n)$.
The eigenvalues are explicitly known only for some very special matrix families from these classes; mainly when the eigenvectors are the columns of the DCT or DST matrices~\cite{BYR2006}.

Over the past decade, there has been an increasing interest in Toeplitz matrices with certain perturbations, see~\cite{BPZ2020,BFGM2014,BYR2006,FK2020,DV2009,GT2009,OA2014,R2017,TS2017,VHB2018,ZJJ2019}, or~\cite{KST1999,NR2019,SM2014,W2008} for more general researches.
In~\cite{FK2020,DV2009} the authors find the characteristic polynomial for some cases of Toeplitz matrices with corner perturbations.
The methods used in the present paper are similar to the ones from~\cite{GMS2021}, where we studied the hermitian tridiagonal Toeplitz matrices with perturbations in the positions $(1,n)$ and $(n,1)$.

The asymptotic distribution of hermitian Toeplitz matrices with small-rank perturbations is described by analogs of Szeg\H{o} theorem~\cite{GS2017,Tilli1998,Tyrtyshnikov1996}.
The individual behavior of the eigenvalues is known only for some particular cases, including hermitian Toeplitz matrices with simple-loop symbols~\cite{BBGM2018,BGM2017,BBGM2015,BBGM2017}.

In~\cite{GMS2021} we studied the eigenvalues of the hermitian tridiagonal Toeplitz matrices with diagonals $-1,2-1$ and values $-\al$ and $-\overline{\al}$ on the corners $(n,1)$ and $(1,n)$, respectively.
In the present paper, we put $1+\al$ instead of $2$ in the entries $(1,1)$ and $(n,n)$.

The matrices $L_{\al,n}$ are real and symmetric, thus their eigenvalues are real. We enumerate them in the ascending order:
\begin{equation}\label{eq:eigvals_order}
\la_{\al,n,1}
\leq\la_{\al,n,2}
\leq\cdots
\leq\la_{\al,n,n}.
\end{equation}
It is well known that every laplacian matrix has eigenvalue $0$ associated to the eigenvector $[1,\ldots,1]^\top$. 

For $\al=0$, the eigenvalues of $L_{0,n}$ are $\la_{0,n,j} = g((j-1)\pi/n)$, where $g$ is defined by \begin{equation}\label{eq:g_main}
    g(x) \eqdef 2-2\cos(x)
        = 4\sin^2\frac{x}{2}, \quad x\in[0,\pi].
\end{equation}
The normalized eigenvectors of $L_{0,n}$ are the columns of the matrix DCT-II, see~\cite[formula (2.53) and (2.54)]{BYR2006}.

For $\al=1$, the matrices $L_{1,n}$ are circulant, and their eigenvalues and eigenvectors are well known, see, e.g.~\cite{GMS2021}. 

It is also well known that the eigenvalues of tridiagonal real symmetric Toeplitz matrices $T_n(g)$ generated by $g$ are $g(j\pi/(n+1))$. 

Except for the cases $\al=0$, $\al=1$, and $\al=1/2$ (see Remark~\ref{rem:case_al=1/2}), we do not know explicit formulas for all eigenvalues of $L_{\al,n}$.

For $\al<0$ (resp., $\al>1$), it can be shown that the first (resp., last) eigenvalue goes out the interval $[0,4]$ and tends exponentially to $4\al^2/(2\al-1)$.
We are going to present the corresponding results in another paper.

In this paper we suppose that $0<\al<1$.

Our matrices $L_{\al,n}$ can be obtained by small-rank perturbations from $T_n(g)$, $L_{0,n}$ or $L_{1,n}$.
The Cauchy interlacing theorem or the theory of locally Toeplitz sequences
\cite{GS2017,Tilli1998,Tyrtyshnikov1996} easily imply that the eigenvalues of $L_{\al,n}$ are asymptotically distributed as the values of $g$ on $[0,\pi]$, as $n$ tends to infinity.

We obtain much more precise results about the eigenvalues of $L_{\al,n}$.
Namely, we find exact eigenvalues of the form $g((j-1) \pi /n)$, with $j$ odd, and localize the other eigenvalues in the intervals of the form $(g((j-1)\pi/n)$, $g(j\pi/n))$ with $j$ even.
 
We transform the characteristic equation to the form $x =  f_{\al,n,j}(x)$, where $f_{\al,n,j}$ is ``slow'', i.e., the derivative of $f_{\al,n,j}$ is small when $n$ is large.
After that, this equation is convenient to solve by the fixed point method and Newton's method (also known as Newton--Raphson or gradient method).

On this base, we derive asymptotic formulas for all eigenvalues $\la_{\al,n,j}$, 
where the errors are uniformly bounded on $j$.

For $\al$ in $\bC$,
we consider the $n\times n$ complex laplacian matrix $L_{\al,n}$, for example,
\begin{equation}
\label{eq:lapl_matrix_complex}
L_{\al,7}
=
\begin{bmatrix*}[r]
1+\overline{\al}& -1 & 0 & 0 & 0 & 0 & -\overline{\al}\phantom{\al} \\
-1\phantom{\al} & 2 & -1 & 0 & 0 & 0 & 0\phantom{\al} \\
0\phantom{\al} & -1 & 2 & -1 & 0 & 0 & 0\phantom{\al} \\
0\phantom{\al} & 0 & -1 & 2 & -1 & 0 & 0\phantom{\al} \\
0\phantom{\al} & 0 & 0 & -1 & 2 & -1 & 0\phantom{\al} \\
0\phantom{\al} & 0 & 0 & 0 & -1 & 2 & -1\phantom{\al} \\
-\al\phantom{\al} & 0 & 0 & 0 &  0 & -1 & 1+\al
\end{bmatrix*}.
\end{equation}
These matrices appear in the study of problems related to networked multi-agent systems,
see~\cite{LWHF2014} for investigations in this area.
In Proposition~\ref{prop:L_char_pol_via_Cheb} we prove that the characteristic polynomial of $L_{\al,n}$
only depends on $\Re(\al)$, i.e., $\det(\la I_n - L_{\al,n}) = \det(\la I_n - L_{\Re(\al),n})$.

We present the main results of this paper in Section~\ref{sec:Main_results}, the correspondent proofs lie in Section~\ref{sec:localization_eigenvalues} (localization), Section~\ref{sec:main_equation} (main equation), Section~\ref{sec:fixed_point} (fixed point method), Sections~\ref{sec:Newton_convex} and~\ref{sec:solve_by_Newton} (Newton's method), Section~\ref{sec:asymptotic_formulas} (asymptotic formulas), Section~\ref{sec:eigvec_norm} (norms of the eigenvectors).
In Section~\ref{sec:trid_toep_corner_per} we give formulas for the characteristic polynomial and eigenvectors of general tridiagonal symmetric Toeplitz matrices with perturbations in the corners $(1,1)$, $(1,n)$, $(n,1)$ and $(n,n)$; our formulas are equivalent to Yueh and Cheng~\cite{YuehCheng2008}.
In Section~\ref{sec:num_exp} we show the results of some numerical tests.

\section{Main results}\label{sec:Main_results}

We treat $\al$ as a fixed parameter, supposing that $0<\al<1$.

It is well known that $0$ is the least eigenvalue of $L_{\al,n}$.
A direct application of the Gershgorin disks theorem~\cite[Theorem 6.1.1]{HornJohnson2013} shows that all eigenvalues of $L_{\al,n}$ belong to $[0,4]$. However, we give a more precise localization. 

\begin{thm}[eigenvalues' localization]\label{thm:Localization_weak_eigenvals}
For every $n\ge3$,
    \begin{align}\label{eq:localization_eigvals_weak_odd}
        \la_{\al,n,j}& = g\left(\frac{(j-1)\pi}{n}\right) \quad
        (j\ \text{odd},\ 1\le j\le n), \\\label{eq:localization_eigvals_weak_even}
        g\left(\frac{(j-1)\pi}{n}\right) &< \la_{\al,n,j}<g\left(\frac{j\pi}{n}\right)\quad
        (j\ \text{even},\ 1\le j\le n).
    \end{align}
\end{thm}

In particular, Theorem~\ref{thm:Localization_weak_eigenvals} implies that $\la_{\al,n,j}$ with odd $j$ does not depend on $\al$.

Motivated by Theorem~\ref{thm:Localization_weak_eigenvals}, we use $g$ as a change of variable in the characteristic equation and put 
\[
d_{n,j}\eqdef\frac{(j-1)\pi}{n},\qquad
\tht_{\al,n,j}\eqdef \widetilde{g}^{-1}(\la_{\al,n,j}),
\]
where $\widetilde{g}\colon[0,\pi]\to[0,4]$ is a restriction of $g$.
In other words, the numbers $\tht_{\al,n,j}$ belong to $[0,\pi]$ and satisfy $g(\tht_{\al,n,j})=\la_{\al,n,j}$.
Then~\eqref{eq:localization_eigvals_weak_odd} and~\eqref{eq:localization_eigvals_weak_even} are equivalent to
\begin{equation*}
    \begin{aligned}
    \,&\tht_{\al,n,j} =
    d_{n,j}\qquad
    & &
    (j\ \text{odd},\ 1\le j\le n),\\[1ex]
    d_{n,j}< \,&\tht_{\al,n,j}<d_{n,j+1}\qquad & &(j\ \text{even},\ 1\le j\le n).
    \end{aligned}
\end{equation*}

We define $\eta_\al\colon[0,\pi]\to\bR$ by
\begin{equation}\label{eq:eta}
 \eta_\al(x)\eqdef
 2\arctan\left(\ka_\al\cot\frac{x}{2}\right),
\end{equation}
where 
\begin{equation}\label{eq:ka}
\ka_\al\eqdef\frac{\al}{1-\al}.
\end{equation}
Obviously, $\eta_\al$ strictly decreases taking values from $\pi$ to $0$. 
Furthermore, $\eta_\al$ is strictly convex when $0<\al<1/2$ and strictly concave if $1/2<\al<1$.
Other equivalent formulas for $\eta_\al$ are given in~\eqref{eq:eta1}, \eqref{eq:eta2}, and \eqref{eq:eta3}.
A direct computation shows that $\eta_\al$ is an involution of the segment $[0,\pi]$, i.e.,
$\eta_\al(\eta_\al(x))=x$ for every $x$ in $[0,\pi]$.
This property is not used in the paper.
See~\cite{Z2014} for the general description of the continuous involutions of real intervals.

\begin{thm}[main equation]\label{thm:weak_characteristic_equation_L}
Let $n\ge3$ and $j$ be even, $1\le j\le n$. Then the number $\tht_{\al,n,j}$
is the unique solution of the following equation on $[0,\pi]$:
\begin{equation}\label{eq:main_eq}
x=d_{n,j}+\frac{\eta_\al(x)}{n}.
\end{equation}
\end{thm}

Figure~\ref{fig:theta_vs_f} shows the left-hand side and the right-hand side of~\eqref{eq:main_eq}.

\begin{figure}[htb]
\centering
\includegraphics{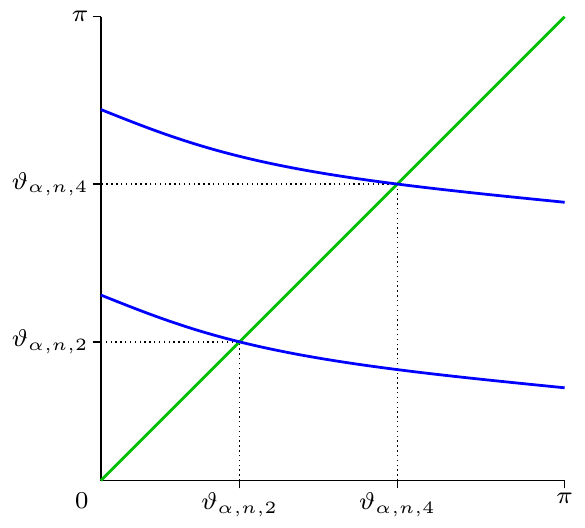}\quad\qquad \includegraphics{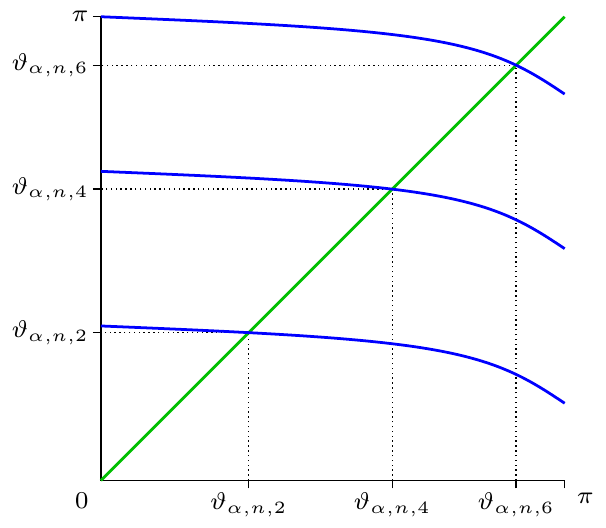}
\caption{The left picture shows the left-hand side (green) and the right-hand side (blue) of~\eqref{eq:main_eq} for $\al=1/3$, $n=5$, $j=2,4$. The right picture corresponds to $\al=4/5$,  $n=6$, $j=2,4,6$. \label{fig:theta_vs_f}}
\end{figure}

The main equation can be rewritten in the form $nx-(j-1)\pi = \eta_\al(x)$.
Figure~\ref{fig:nx_plus_jpi_eq_eta} shows both sides of this equation for some values of $\al$ and $n$, $j$.

For every $j$ with $1\le j\le n$, we define $I_{n,j} \eqdef \left(\frac{(j-1)\pi}{n}, \frac{j\pi}{n}\right)$.

\begin{figure}[hbt]
\centering
\includegraphics{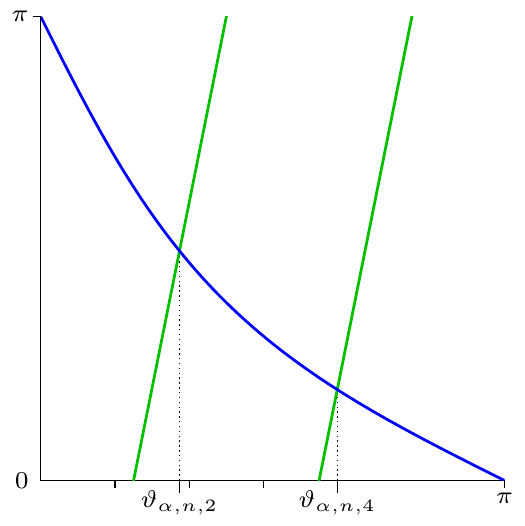}\quad\qquad\includegraphics{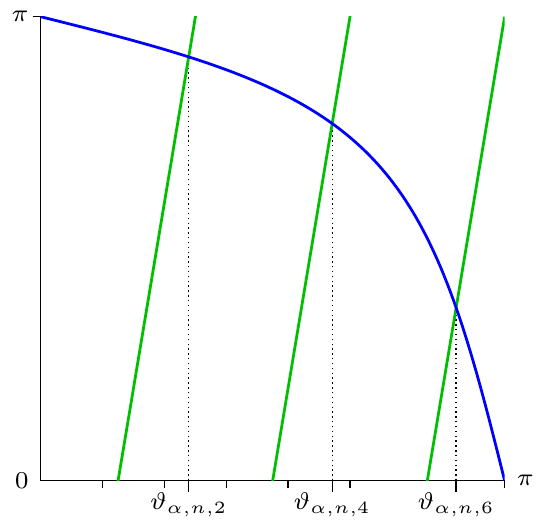}
\caption{Plots of $x\mapsto nx - (j-1)\pi$ (green) and $\eta_\al$ (blue), for $\al=1/3$, $n=5$ (left) and $\al=4/5$, $n=6$ (right).
\label{fig:nx_plus_jpi_eq_eta}
}
\end{figure}

For every $n\ge3$ and every $j$ even with $1\le j\le n$, we define $h_{\al,n,j}\colon \clos(I_{n,j})\to \bR$ by
\begin{equation}\label{eq:h_char}
    h_{\al,n,j}(x)\eqdef nx-(j-1)\pi-\eta_\al(x).
\end{equation}
In Proposition~\ref{prop:h_change_sign} we show that $h_{\al,n,j}$ changes its sign in $I_{n,j}$.
Hence, it is feasible to solve~\eqref{eq:main_eq} by the bisection method or false rule method.

In Proposition~\ref{prop:dependence_on_the_parameter} we study the dependence of $\la_{\al,n,j}$ on the parameter $\al$ (if $n$ and $j$ are fixed).

Proposition~\ref{prop:Z_contractive_weak} states that if $n$ is large enough, then the functions $x\mapsto d_{n,j}+\eta_\al(x)/n$ are contractive and the fixed-point method yields the solution of~\eqref{eq:main_eq}. 

Moreover, surprisingly for us, Newton's method applied to the equation $h_{\al,n,j}(x) = 0$ converges for \emph{all} $n\ge3$.

\begin{thm}[convergence of Newton's method]
\label{thm:Newton}
Let $n\ge3$, $j$ be even, $1\le j\le n$
and $y_{\al,n,j}^{(0)} \in\clos(I_{n,j})$.
Define the sequence $(y_{\al,n,j}^{(m)})_{m=0}^\infty$ by the recursive formula
\begin{equation}\label{eq:Newton_sequence_eigval} y_{\al,n,j}^{(m)}
\eqdef y_{\al,n,j}^{(m-1)} - \frac{h_{\al,n,j}\left(y_{\al,n,j}^{(m-1)}\right)}{h_{\al,n,j}'\left(y_{\al,n,j}^{(m-1)}\right)}\quad (m\ge1).
\end{equation}
Then $(y_{\al,n,j}^{(m)})_{m=0}^\infty$ converges to
$\tht_{\al,n,j}$. 
If $n>\sqrt{\pi\cK_2(\al)/2}$, then for every $m$
\begin{equation}
\label{eq:sol_newton_convergence_eigval}
    \left|y_{\al,n,j}^{(m)} - \tht_{\al,n,j}\right| \le \frac{\pi}{n}\left(\frac{\pi \cK_2(\al)}{2n^2}\right)^{2^m-1}.
\end{equation}
\end{thm}

We define $\Lambda_{\al,n}\colon [0,\pi]\to\bR$ by
\begin{equation}\label{eq:La}
    \Lambda_{\al,n}(x) \eqdef g(x)
+\frac{g'(x)
\eta_\al(x)}{n}
+\frac{g'(x)
\eta_\al(x)\eta_\al'(x)
+\frac{1}{2}g''(x)
\eta_\al(x)^2}{n^2}.
\end{equation}
For $j$ even, $1\le j\le n$,
we define $\laasympt_{\al,n,j}$ by
\begin{equation}\label{eq:laasympt_w}
\laasympt_{\al,n,j}
\eqdef \Lambda_{\al,n}\left(d_{n,j}\right).
\end{equation}

\begin{thm}[asymptotic expansion of the eigenvalues]
\label{thm:weak_asympt_weak}
There exists $C_1(\al)>0$
such that for $n$ large enough and $j$ even, $1\le j\le n$,
\begin{equation}\label{eq:weak_lambda_asympt_K}
\left|\la_{\al,n,j}-\laasympt_{\al,n,j}\right|
\le\frac{C_1(\al)}{n^3}.
\end{equation}
\end{thm}

The asymptotic expansion~\eqref{eq:weak_lambda_asympt_K} can be written as $\la_{\al,n,j}=\Lambda_{\al,n}(d_{n,j})+O_\al(1/n^3)$, where the constant $C_1(\al)$ in the upper bound of $O_\al(1/n^3)$ depends on $\al$, but does not depend on $j$ or $n$.

Proposition~\ref{prop:weak_asympt_2}
gives an alternative asymptotic expansion for $\la_{\al,n,j}$, with the points $j\pi/(n+1)$ instead of $d_{n,j}$.

Proposition~\ref{prop:first_eigenvalues} contains an asymptotic expansion of $\la_{\al,n,j}$ for small values of $j$, as $j/n$ tends to $0$.
Notice that $\la_{\al,n,2}$ is the first non-zero eigenvalue of $L_{\al,n}$ and is known as the ``spectral gap'' of this matrix.

In the upcoming theorem we show an explicit formula~\eqref{eq:eivec_w} for the eigenvectors of $L_{\al,n}$ and asymptotic formulas for their norms; in these results we extend the domain of $\al$ to the strip $0<\Re(\al) <1$ of the complex plane, see~\eqref{eq:lapl_matrix_complex}.
In the complex case we define $\ka_\al$ as $\Re(\al)/(1-\Re(\al))$. 
Formula~\eqref{eq:eivec_w} is a particular case of~\cite[Theorem 3.1]{YuehCheng2008}.

For every $x$ in $[0,\pi]$, we define
\begin{equation}\label{eq:nu_al}
 \nu_{\al}(x) \eqdef \frac{1-\Re(\al)}{2}g(x) - \frac{\Re(\al)}{2} g(\eta_\al(x)) + \frac{\Re(\al)-|\al|^2}{2}g(x-\eta_\al(x)) + 2|\al|^2.
\end{equation}

\begin{thm}[eigenvectors and their norms]\label{thm:norm_eigvec}
Let $\al\in\bC$, $0<\Re(\al)<1$. Then the vector $[1,\ldots,1]^\top$ is an eigenvector of the matrix $L_{\al,n}$ associated to the eigenvalue $\la_{\al,n,1}=0$.
For every $j$, $2\le j\le n$, and every $k$, $1\le k\le n$, we define
\begin{equation}\label{eq:eivec_w}
    v_{\al,n,j,k}\eqdef \sin(k \tht_{\al,n,j}) -(1-\overline{\al}) \sin((k-1)\tht_{\al,n,j})  + \overline{\al} \sin((n-k)\tht_{\al,n,j}).
\end{equation}
Then the vector $v_{\al,n,j}=[v_{\al,n,j,k}]_{k=1}^n$ with components~\eqref{eq:eivec_w}
is an eigenvector of $L_{\al,n}$ associated to $\la_{\al,n,j}$.
Moreover, if $j$ is odd, then
\begin{equation}\label{eq:norm_eigvec_j_odd}
    \|v_{\al,n,j}\|_2 = |1-\al|\sqrt{\frac{n}{2} \la_{\al,n,j}}.
\end{equation}
If $j$ is even, then
\begin{equation}\label{eq:norm_eigvec_j_even}
    \|v_{\al,n,j}\|_2 = \sqrt{n\nu_\al\left(\tht_{\al,n,j}\right)} + O_\al\left(\frac{1}{\sqrt{n}}\right) ,
\end{equation}
with $O_\al\left(\frac{1}{\sqrt{n}}\right)$ uniformly on $j$.
\end{thm}

\section{Tridiagonal Toeplitz matrices with corner perturbations}\label{sec:trid_toep_corner_per}

Let $\de$, $\eps$, $\si$, $\tau$ be arbitrary complex parameters and $n\ge3$.
In this section, we consider the $n\times n$ matrix $A_n$, obtained from the tridiagonal Toeplitz matrix  with diagonals $-1$, $2$, $-1$, substituting the components $(1,1)$, $(1,n)$, $(n,1)$, and $(n,n)$ by  $2-\de$, $-\eps$, $-\si$, and $2-\tau$, respectively.
For example,
\begin{equation}\label{eq:Toeplitz_corner_pert}
A_6\eqdef\begin{bmatrix*}[r]
2-\de & -1 & 0 & 0 & 0 & -\eps\phantom{\de} \\
-1\phantom{\de} & 2 & -1 & 0 & 0 & 0\phantom{\de} \\
0\phantom{\de} & -1 & 2 & -1 & 0 & 0\phantom{\de} \\
0\phantom{\de} & 0 & -1 & 2 & -1 & 0\phantom{\de} \\
0\phantom{\de} & 0 & 0 & -1 & 2 & -1\phantom{\de} \\
-\si\phantom{\de} & 0 & 0 & 0 & -1 & 2-\tau
\end{bmatrix*}.
\end{equation}
The study of more general tridiagonal symmetric Toeplitz matrices (with diagonals $a_1$, $a_0$, $a_1$ instead of $-1$,$2$,$-1$) with corner perturbations can be easily reduced to this case.

We are going to give formulas for the characteristic polynomial and eigenvectors of $A_n$.
The results are not essentially new (see~\cite{Ferguson1980,FF2009,YuehCheng2008}), but we present them in a different form (with Chebyshev polynomials) and with other proofs.

We put $D_n(\la)\eqdef\det(\la I_n-A)$ and denote by $T_n$ and $U_n$ the Chebyshev polynomials of degree $n$ of the first and second kind, respectively.
The next proposition is a particular case of~\cite[Corollary 2.4]{FF2009}; it is also easy to prove directly expanding by cofactors.

\begin{prop}[the characteristic polynomial of $A_n$]
\label{prop:char_pol_trid_toep_pert}
\begin{equation}\label{eq:char_pol_trid_toep_pert}
\begin{aligned}
    D_n(\la) 
    & = U_n\left(\frac{\la-2}{2}\right)  + (\de+ \tau) U_{n-1}\left(\frac{\la-2}{2}\right) \\&\pheq
    +(\de\tau - \eps \si) U_{n-2} \left( \frac{\la-2}{2}\right) + (-1)^{n+1}(\eps + \si).
\end{aligned}
\end{equation}
\end{prop}

\begin{cor}
If $\eps = \de$ and $\si = \tau = -\de$, then $ D_n(\la) = U_n\left((\la-2)/2\right)$.
Therefore, the eigenvalues of $A_n$ are $g(j\pi/(n+1))$  with $j$ in $\{1,\ldots,n\}$. The same situation holds for $\si = \de$ and $\eps = \tau = -\de$.
\end{cor}

If $\la$ is an eigenvalue of $A_n$, we will search for an associated eigenvector $v = [v_k]_{k=1}^n$ as a linear combination of two geometric progressions:
\begin{equation}\label{eq:v}
    v_k = G_1 z^k + G_2 z^{-k} \qquad (1\le k\le n),
\end{equation}
where
$z$ is a solution of
the quadratic equation $z^2+(\la-2)z +1 =0$.
Equivalently, $\la$ and $z$ are related by
\begin{equation}\label{eq:la_z}
    -z^{-1} + (2-\la)
    - z = 0.
\end{equation}
Let $w\eqdef(\la I_n - A_n)v$.
Formulas~\eqref{eq:v} and~\eqref{eq:la_z} easily imply that $w_k = 0$ for $2\le k\le n-1$, and our goal is to find coefficients $G_1$ and $G_2$ such that $w_1=0$ and $w_n=0$.

To take advantage of the symmetry between $z$ and $z^{-1}$, we rewrite~\eqref{eq:v} in terms of Chebyshev polynomials:
\begin{equation*}
\begin{aligned}
    v_k &= \left(\frac{G_1+G_2}{2}\right) (z^k + z^{-k}) + \left(\frac{G_1-G_2}{2}\right) (z^k - z^{-k})\\
    &= (G_1+G_2) T_k\left(\frac{z+z^{-1}}{2}\right) + \frac{(G_1-G_2)(z-z^{-1})}{2} U_{k-1}\left(\frac{z+z^{-1}}{2}\right).
\end{aligned}
\end{equation*}
The system $w_1=0$ and $w_n=0$ is equivalent to
\begin{equation}\label{eq:system_x_y}
\begin{aligned}
    \rmA_{\de,\eps,n} x + \rmB_{\de,\eps,n} y & = 0\\
    \rmC_{\si,\tau,n} x + \rmD_{\si,\tau,n} y & = 0,
\end{aligned}
\end{equation}
where $x \eqdef (G_1+G_2)/2$, $y\eqdef (G_1-G_2)/2$, and
\begin{equation}\label{eq:ABCD}
\begin{aligned}
    \rmA_{\de,\eps,n} & \eqdef 2\left(-1 + \de T_1\left(\frac{z+z^{-1}}{2}\right)  + \eps T_n\left(\frac{z+z^{-1}}{2}\right) \right), 
    \\
    \rmB_{\de,\eps,n} & \eqdef (z-z^{-1}) \left( \de + \eps U_{n-1}\left(\frac{z+z^{-1}}{2}\right) \right),
    \\
    \rmC_{\si,\tau,n} & \eqdef 2\left(\si T_1\left(\frac{z+z^{-1}}{2}\right)+ \tau T_n\left(\frac{z+z^{-1}}{2}\right) -  T_{n+1}\left(\frac{z+z^{-1}}{2}\right) \right), \\
    \rmD_{\si,\tau,n} & \eqdef (z- z^{-1})\left( \si +  \tau U_{n-1}\left(\frac{z+z^{-1}}{2}\right) - U_n\left(\frac{z+z^{-1}}{2}\right) \right).
\end{aligned}      
\end{equation}

In the next proposition we use the convention that $U_{-1}(t) \eqdef 0$.

\begin{prop}[eigenvectors of $A_n$]\label{prop:eigvec_tri_Toep_corner_per}
Let $\la\in\bC\setminus\{0,4\}$ be an eigenvalue of $A_n$.
If $\rmA_{\de,\eps,n}\neq0$ or $\rmB_{\de,\eps,n}\neq0$, then the vector $v = [v_k]_{k=1}^n$ with components
\begin{equation}\label{eq:eigvec_AB_neq_0}
v_k \eqdef (-1)^{k-1}\left(U_{k-1}\left(\frac{\la-2}{2}\right) + \de U_{k-2}\left(\frac{\la-2}{2}\right) + (-1)^n \eps U_{n-k-1}\left(\frac{\la-2}{2}\right)\right)
 \end{equation}
is an eigenvector of $A_n$ associated to $\la$. If $\rmC_{\si,\tau,n}\neq0$ or $\rmD_{\si,\tau,n}\neq0$, then the vector $v 
= [v_k]_{k=1}^n$ with components
\begin{equation}\label{eq:eigvec_CD_neq_0}
v_k \eqdef (-1)^{k-1}\left(\si U_{k-2}\left(\frac{\la-2}{2}\right) + (-1)^n \tau U_{n-k-1}\left(\frac{\la-2}{2}\right) + (-1)^n U_{n-k}\left(\frac{\la-2}{2}\right)\right)
\end{equation}
is an eigenvector of $A_n$ associated to $\la$.
\end{prop}
\begin{proof}
The assumptions $\la\notin\{0,4\}$ and $z+z^{-1} = 2-\la$ imply that $z\notin\{-1,1\}$ and
\begin{equation}\label{eq:Cheby_pol_z}
    T_n\left(\frac{\la-2}{2}\right) = (-1)^n \frac{z^n+z^{-n}}{2}, \qquad U_n\left(\frac{\la-2}{2}\right) = (-1)^{n}\frac{z^{n+1} - z^{-(n+1)}}{z-z^{-1}}.
\end{equation}
A direct computation shows that
\begin{equation}
    \rmA_{\de,\eps,n} \rmD_{\si,\tau,n} - \rmB_{\de,\eps,n} \rmC_{\si,\tau,n} = 2 (-1)^n (z-z^{-1}) D_n(\la).
\end{equation}
Since $\la$ is an eigenvalue of $A_n$, we get $D_n(\la) = 0$, and the linear homogeneous system~\eqref{eq:system_x_y} has non-trivial solutions $(x,y)$.
Namely, if $\rmA_{\de,\eps}\neq0$ or $\rmB_{\de,\eps}\neq0$, we put 
\[ x = \frac{\rmB_{\de,\eps}}{2(z-z^{-1})}, \qquad y = -\frac{\rmA_{\de,\eps}}{2(z-z^{-1})} .
\]
Using~\eqref{eq:Cheby_pol_z} we simplify $G_1$ and $G_2$ to
 \begin{align*}
     G_1  &=x+y = \frac{\rmB_{\de,\eps,n}-\rmA_{\de,\eps,n}}{2(z-z^{-1})} = \frac{1 - \de z^{-1} - \eps z^{-n}}{z-z^{-1}},\\ 
     G_2  &=x-y = \frac{\rmB_{\de,\eps,n}+\rmA_{\de,\eps,n}}{2} = \frac{-1 +  \de z + \eps z^n}{z-z^{-1}}.
 \end{align*}
Hence, for every $k$, formula~\eqref{eq:v} converts in
 \begin{equation}\label{eq:v_comp_temp}
    v_k  = \frac{z^k - z^{-k}}{z-z^{-1}} - \de\, \frac{z^{k-1} - z^{-(k-1)}}{z-z^{-1}} + \eps\, \frac{z^{n-k} - z^{-(n-k)}}{z-z^{-1}},
\end{equation}
which by~\eqref{eq:Cheby_pol_z} simplifies to~\eqref{eq:eigvec_AB_neq_0}. The linear independence of the geometric progressions $[z^{k}]_{k=1}^n$ and $[z^{-k}]_{k=1}^{n}$ assures that $v$ is a non-zero vector. 
The proof of~\eqref{eq:eigvec_CD_neq_0} is similar.
\end{proof}

Proposition~\ref{prop:eigvec_tri_Toep_corner_per} does not cover the situation when
\begin{equation}\label{eq:abcd=0}
\rmA_{\de,\eps,n} = \rmB_{\de,\eps,n} = \rmC_{\si,\tau,n} = \rmD_{\si,\tau,n}=0.
\end{equation}
We analyze this situation in the following remarks.

\begin{remark}
If $\la=0$, i.e., $z=1$, then~\eqref{eq:abcd=0}
is equivalent to $\de+\eps = 1$ and $\si+\tau = 1$. The last two equalities imply that $A_n$ is a laplacian complex matrix and $v=[1]_{k=1}^n$ is an eigenvector associated to $\la$.
\end{remark}

\begin{remark}
If $\la = 4$, i.e., $z=-1$, then~\eqref{eq:abcd=0} is equivalent to $\de + (-1)^n\eps = 1$ and $(-1)^n\si - \tau = 1$. If these conditions are fulfilled, $v=[(-1)^k]_{k=1}^n$ is an eigenvector associated to $\la$.
\end{remark}

\begin{remark}
If $\la\notin\{0,4\}$, then~\eqref{eq:abcd=0} is equivalent to
\[ \de = \tau = \frac{U_{n-1}\left(\frac{z+z^{-1}}{2}\right)}{U_{n-2}\left(\frac{z+z^{-1}}{2}\right)}, \qquad \eps = \si = -\frac{1}{U_{n-2}\left(\frac{z+z^{-1}}{2}\right)}. \]
In this case, every vector with components of the form~\eqref{eq:v} belongs to $\ker(\la I_n - A_n)$, and $\la$ is an eigenvalue of multiplicity at least $2$.
\end{remark}

\begin{remark}
We have tested most formulas of this section in Sagemath using symbolic computations with polynomials over the variables $\de,\eps,\si,\tau,\la$,
for every $n$ with $3\le n\le 20$.
In particular, we have verified
that if $v$ is given by~\eqref{eq:eigvec_AB_neq_0} and $w = (\la I_n -A_n)v$, then $w_n = (-1)^{n+1} D_n(\la)$.
Analogously, if $v$ is given by~\eqref{eq:eigvec_CD_neq_0}, then $ w_1 = (-1)^{n} D_n(\la)$.
\end{remark}

\section{Eigenvalues' localization}
\label{sec:localization_eigenvalues}
 
In the incoming proposition, unlike the main part of the paper, we suppose that $\al$ is a complex parameter.
We define $D_{\al,n}(\la)$ as the characteristic polynomial
$\det (\la I_n - L_{\al,n})$,
where $L_{\al,n}$ is the $n\times n$ complex laplacian matrix of the form~\eqref{eq:lapl_matrix_complex}.

\begin{prop}[characteristic polynomial of complex laplacian matrices]
\label{prop:L_char_pol_via_Cheb}
For $n\ge 3$,
\begin{equation}\label{eq:L_char_pol_via_Cheb}
D_{\al,n}(\la)
= (\la-2\Re(\al)) U_{n-1}\left(\frac{\la-2}{2}\right)
-2\Re(\al) U_{n-2}\left(\frac{\la-2}{2}\right)+2(-1)^{n+1}\Re(\al).
\end{equation}
\end{prop}
\begin{proof}
This is a corollary of Proposition~\ref{prop:char_pol_trid_toep_pert}.
\end{proof}

Formula~\eqref{eq:L_char_pol_via_Cheb} implies a little miracle: $ D_{\al,n} = D_{\Re(\al),n}$ for every complex $\al$.
Therefore, the eigenvalues of $L_{\al,n}$ are the same as the ones of the matrix $L_{\Re(\al),n}$. Since the latter matrix is hermitian, the eigenvalues are real.
Hence, from now on we will suppose $\al$ to be a real number.

It turns out that $D_{\al,n}(\la)$ factorizes into a product of two polynomials of nearly the same degree.
To join the cases when $n$ is even and $n$ is odd, we use the change of variables $\la =4-t^2$.

\begin{prop}\label{prop:pol_char_factorized}
For $n\ge 3$,
\begin{equation}\label{eq:char_pol_fact}
D_{\al,n}(4-t^2) = 2(-1)^n \frac{p_{n}(t) q_{\al,n}(t)}{t},
\end{equation}
where
\[ 
p_{n}(t) = (t^2-4)
U_{n-1}\left(\frac{t}{2}\right), \qquad q_{\al,n}(t) =  (1-\al)T_n\left(\frac{t}{2}\right)
+\al \frac{t}{2} U_{n-1}\left(\frac{t}{2}\right).
\]
\end{prop}
\begin{proof}
We will give a proof only for the case  $n=2m$. 
The case $n=2m+1$ is similar.
First, put $\la=2\om+2$, hence $t^2 = 2 - 2\om$. 
We apply the following elementary relations for Chebyshev polynomials: 
\begin{align*}
    U_{2m-2}(\om) &= -U_{2m}(\om) + 2\om U_{2m-1}(\om), \\
    U_{2m-1}(\om) &=2U_{m-1}(\om) T_{m}(\om) , \\
    U_{2m}(\om) &= 2\om U_{m-1}(\om)T_{m}(\om) +2T_m^2(\om) -1, \\
    T_m^2(\om) - 1 &= (\om^2-1) U_{m-1}^2(\om), \\
     T_{2m}\left(\frac{t}{2}\right) & = T_m\left(\frac{t^2-2}{2}\right),\quad  U_{2m+1}\left(\frac{t}{2}\right) = t U_m\left(\frac{t^2-2}{2}\right).
\end{align*}
Thereby we obtain the next chain of equalities:
\begin{align*}
    D_{\al,2m}(2\om+2) & = 2\Bigl(\al U_{2m}(\om) + (\om+1-\al-2\al \om)U_{2m-1}(\om) - \al\Bigr) \\
    & = 4\Bigl( (\om+1)(1 - \al) U_{m-1}(\om) T_m(\om)  + \al (T_m^2(\om)-1)\Bigr) \\
    & = 4 (\om +1) U_{m-1}\left(\om\right) \Bigl((1-\al)T_{m}\left(\om\right) - \al (1-\om) U_{m-1 }\left(\om\right)\Bigr),
\end{align*}
and we arrive at~\eqref{eq:char_pol_fact}.
\end{proof}

The factorization~\eqref{eq:char_pol_fact} after the change of variable $t = 2\cos(x/2)$ reads as
\begin{equation}\label{eq:char_pol_fact_trig} D_{\al,n}(g(x)) = D_{\al,n}(4-(2\cos(x/2))^2) = (-1)^n \frac{p_{n}(2\cos(x/2))q_{\al,n}(2\cos(x/2))}{\cos(x/2)}, \end{equation}
where
\[ p_n(2\cos(x/2)) = -4\sin\frac{x}{2} \sin\frac{nx}{2},  \qquad q_{\al,n}(2\cos(x/2)) = (1-\al) \cos\frac{nx}{2} + \al\cos\frac{x}{2} \frac{\sin\frac{nx}{2}}{\sin\frac{x}{2}}, \]
or 
\begin{equation}
\label{eq:charpol_factorization_trig}
D_{\al,n}(g(x)) = (-1)^{n+1} \frac{4\sin\frac{x}{2}\sin\frac{nx}{2}}{\cos\frac{x}{2}} \left( (1-\al)\cos\frac{nx}{2} + \al \cos\frac{x}{2} \frac{\sin\frac{nx}{2}}{\sin\frac{x}{2}}\right).
\end{equation}
The polynomial $p_{n}$ does not depend on $\al$, and its zeros are easy to find.

\begin{prop}[trivial eigenvalues of $L_{\al,n}$]\label{prop:trivial_eigvals}
For every $n\ge3$ and every even $k$ with $0\le k\le n-1$, the number $g(k\pi/n)$ is an eigenvalue of $L_{\al,n}$.
\end{prop}
\begin{proof}
The number $t = 2\cos(k\pi/(2n))$, with $k$ as in the hypothesis, is a  zero of $p_n$.
It corresponds to the eigenvalue $\la = 4-t^2 = g(k\pi/n)$, since $g(x) = 4-(2\cos(x/2))^2$.
\end{proof}

We already have an explicit formula for $\lfloor (n+1)/2 \rfloor$ eigenvalues of $L_{\al,n}$. 
The remaining ones correspond to the zeros of the polynomial $q_{\al,n}$.
To analyze their localization, we first compute the values of $q_{\al,n}$ at the points $2\cos(j\pi/(2n))$ which correspond to the uniform mesh $j\pi/n$, $j=0,\ldots, n$.

The next lemma is easily proven by direct computations.

\begin{lem}\label{lem:q_n_eval_jpi/n}
For every $j$ with $1\le j\le n-1$,
\begin{equation*}
    q_{\al,n}\left(2\cos\frac{j\pi}{2n}\right) = \begin{cases}
        (1-\al) (-1)^{\frac{j}{2}}, & \quad \text{if } j \text{ is even}, \\
        \al\cot\frac{j\pi}{2n} (-1)^{\frac{j-1}{2}}, & \quad \text{if } j \text{ is odd}.
    \end{cases}
\end{equation*}
Moreover,
\begin{equation*}
    q_{\al,n}(0) = 
    \begin{cases}
        0, & \quad \text{if}\ n\ \text{is odd}, \\
        (-1)^{\frac{n}{2}} (1-\al), &\quad \text{if}\ n\ \text{is even},
    \end{cases}
     \qquad q_{\al,n}(2) = (1-\al) + \al n.
\end{equation*}
\end{lem}

We observe that if $n$ is even, then $p_n(0)=0$, and if $n$ is odd, then $q_{\al,n}(0) = 0$. However, $t=0$ may not be a zero of $D_{\al,n}(4-t^2)$ because of the factor $1/t$ in~\eqref{eq:char_pol_fact}. 
This leads us to the next elementary lemma.

\begin{lem}\label{lem:limits_pq}
If $n$ is odd, then
\begin{equation*}
    \lim_{t\to 0^+} \frac{2q_{\al,n}(t)}{t} = (-1)^{\frac{n-1}{2}} \Bigl( \al + (1-\al)n \Bigr) ,
\end{equation*}
and if $n$ is even, then
\begin{equation*}
    \lim_{t\to 0^+} \frac{2p_{n}(t)}{t} = 4(-1)^{\frac{n}{2}} n.
\end{equation*}
\end{lem}

\begin{proof}[Proof of Theorem~\ref{thm:Localization_weak_eigenvals}]
Let $1\le j\le n$. If $j$ is odd, then~\eqref{eq:localization_eigvals_weak_odd} follows by Proposition~\ref{prop:trivial_eigvals}.

We consider the quotient $q_{\al,n}(2\cos(x/2))/(2\cos(x/2))$ from factorization~\eqref{eq:char_pol_fact_trig}. 
Lemmas~\ref{lem:q_n_eval_jpi/n} and~\ref{lem:limits_pq} imply that this expression changes its sign in the intervals $I_{n,j}$, where $j$ is even. By the intermediate value theorem, we have~\eqref{eq:localization_eigvals_weak_even}.
\end{proof}

Theorem~\ref{thm:Localization_weak_eigenvals} implies immediately that for every $0<\al<1$ and for every $y$ in $\bR$,
\begin{equation*}
    \lim_{n\to\infty} \frac{\#\{j\in\{1,\ldots,n\}\colon \ \la_{\al,n,j}\le y \}}{n}  = \frac{\mu\left(\{x\in[0,\pi]\colon g(x)\le y \}\right)}{\pi},
\end{equation*}
i.e., the eigenvalues of $L_{\al,n}$ are asymptotically distributed as the function $g$ on $[0,\pi]$.

\section{Main equation}
\label{sec:main_equation}

In this section we reduce the computation of the non-trivial eigenvalues to the solution of the ``main equation''~\eqref{eq:main_eq}.
We recall it here:
\[
x = d_{n,j}+\frac{\eta_\al(x)}{n}.
\]

\begin{proof}[Proof of Theorem~\ref{thm:weak_characteristic_equation_L}]
Recall that $j$ is even.
In the proof of Theorem~\ref{thm:Localization_weak_eigenvals} we have seen that $\tht_{\al,n,j}$ belongs to $I_{n,j}$ and is the unique solution of the equation $q_{\al,n}(2\cos(x/2)) = 0$. This is equivalent to the following one (see also~\eqref{eq:char_pol_fact_trig}):
\begin{equation}\label{eq:tangent_equality_L}
    \tan\frac{nx}{2} = -\frac{1-\al}{\al}\tan\frac{x}{2}.
\end{equation}
Applying $\arctan$ to both sides of~\eqref{eq:tangent_equality_L} we transform it to
\[ 
nx = j\pi - 2\arctan\left(\frac{1-\al}{\al}\tan\frac{x}{2}\right). 
\]
Finally, since $\pi/2-\arctan(u) = \arctan(1/u)$, we obtain~\eqref{eq:main_eq}. 
\end{proof}

Figure~\ref{fig:tangents_weak_L} shows the plots of both sides of~\eqref{eq:tangent_equality_L} for some $\al$ in $(0,1)$.
We see that the intersections really take place in the intervals given in Theorem~\ref{thm:Localization_weak_eigenvals}.

\begin{figure}[htb]
\centering
\includegraphics{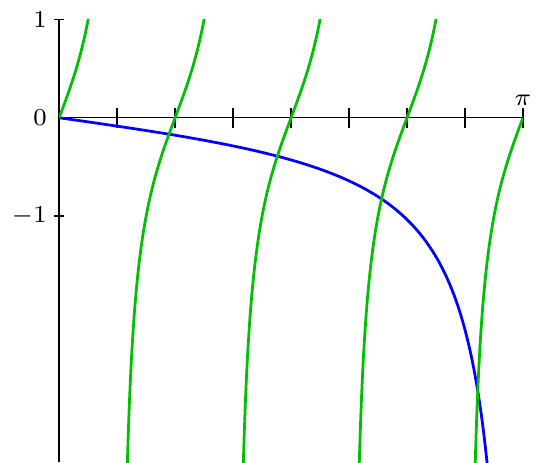}
\caption{The left-hand side (green) and right-hand side (blue) of~\eqref{eq:tangent_equality_L} 
for $\al=0.7$ and $n=8$;
the scales of the axes are different.
\label{fig:tangents_weak_L}
}
\end{figure}

Recall that $h_{\al,n,j}$ is defined by~\eqref{eq:h_char}.
Obviously,~\eqref{eq:main_eq} is equivalent to $h_{\al,n,j}(x)=0$.

\begin{prop}\label{prop:h_change_sign}
Let $n\ge3$ and $j$ be even with $1\le j\le n$.
Then $h_{\al,n,j}$ changes its sign exactly once in $I_{n,j}$.
\end{prop}
\begin{proof}
Indeed,
\[
h_{\al,n,j}((j-1)\pi/n) = - \eta_\al((j-1)\pi/n)<0,
\]
\[
h_{\al,n,j}(j\pi/n) = \pi-\eta_\al(j\pi/n)>0,
\]
and $h_{\al,n,j}$ is strictly increasing.
\end{proof}

\begin{remark}\label{rem:case_al=1/2}
If $\al = 1/2$, then $\ka_{\frac{1}{2}} = 1$ and $\eta_{\al}(x) = \pi -x$.
In this case equation~\eqref{eq:main_eq} yields explicit formulas for the eigenvalues $\la_{\al,n,j}$ with even values of $j$:
\[
\tht_{\al,n,j}
= \frac{j\pi}{n+1},\qquad
\la_{\al,n,j}
= g\left(\frac{j\pi}{n+1}\right).
\]
\end{remark}

In the following proposition, unlike in the other parts of this paper, we fix $n$ and $j$ and treat $\al$ as a variable running through the closed interval $[0,1]$.
Formally, we define $\Psi_{n,j}\colon[0,1]\to[0,4]$ by
\[
\Psi_{n,j}(\al)\eqdef\la_{\al,n,j}.
\]

\begin{prop}[dependence of the eigenvalues on the parameter $\al$]
\label{prop:dependence_on_the_parameter}
Let $n\ge 3$ and $j$ be even, with $1\le j\le n$.
Then $\Psi_{n,j}$ is continuous and strictly increasing on $[0,1]$.
In particular,
\begin{align}
\label{eq:lim_lambda_at_alpha_zero}
\lim_{\al\to0^+}
\la_{\al,n,j}
&=\la_{0,n,j}
=g\left(\frac{(j-1)\pi}{n}\right),
\\[1ex]
\label{eq:lim_lambda_at_alpha_one}
\lim_{\al\to1^-}
\la_{\al,n,j}
&=\la_{1,n,j}
=g\left(\frac{j\pi}{n}\right).
\end{align}
\end{prop}

\begin{proof}
It is well known that the functions $A\mapsto \la_j(A)$ are Lipschitz continuous on the space of the hermitian matrices provided with the operator norm, see~\cite[Weyl's Theorem~4.3.1 and Problem~4.3.P1]{HornJohnson2013}. 
As a consequence, $\Psi_{n,j}$ is continuous on $[0,1]$.

To analyze the monotonicity, we will apply to the main equation
some ideas from the implicity function theorem.
Define $\Tht_{n,j}\colon(0,\pi)\to\bR$ and $H_{n,j}\colon (0,1)\times (0,\pi)\to\bR$ by
\[
\Tht_{n,j}(\al)\eqdef\tht_{\al,n,j},\qquad
H_{n,j}(\al,x)
\eqdef h_{\al,n,j}(x)
= nx - (j-1)\pi - \eta_\al(x).
\]
Compute the partial derivatives of $H_{n,j}$ with respect to the first and second argument:
\[
(D_1 H_{n,j})(\al,x)
= -\frac{2\tan\frac{x}{2}}%
{\al^2+(1-\al)^2\tan^2\frac{x}{2}}
<0,\qquad
(D_2 H_{n,j})(\al,x)
=n-\eta_\al'(x)>n.\]
Since $H_{n,j}(\al,\Tht_{n,j}(\al))=0$, we conclude that $\Tht_{n,j}$ is differentiable on $(0,1)$, and
\[
\Tht_{n,j}'(\al)
=-\frac{(D_1 H_{n,j})(\al,\Tht_{n,j}(\al))}{(D_2 H_{n,j})(\al,\Tht_{n,j}(\al))}>0.
\]
Hence, the functions $\Tht_{n,j}$ and $\Psi_{n,j}=g\circ\Tht_{n,j}$ are strictly increasing on $(0,1)$.
Now the continuity of $\Psi_{n,j}$ implies that this function is strictly increasing on $[0,1]$.
\end{proof}

Figure~\ref{fig:theta_vs_lambda} shows the eigenvalues $\la_{\al,n,j} = g(\tht_{\al,n,j})$ for $\al=1/3$ and $\al=4/5$, with $n=10$.
One can observe the localization of $\tht_{\al,n,j}$ in $I_{n,j}$ for even values of $j$ and the monotone dependence on $\al$.

\begin{figure}[htb]
\centering
\includegraphics{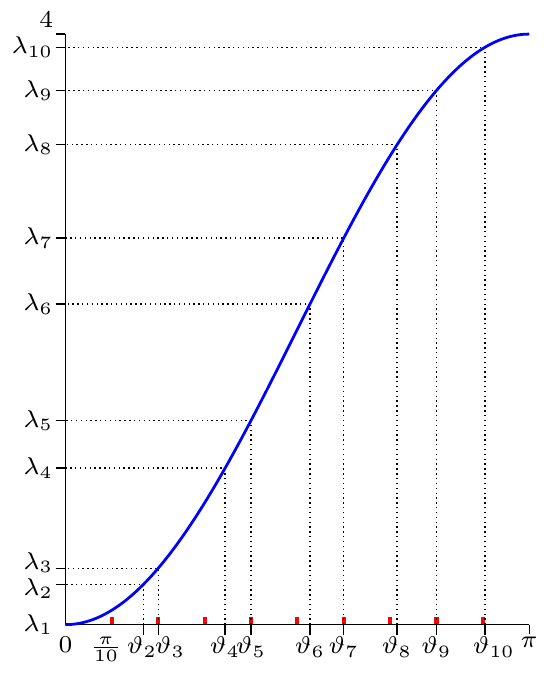}\qquad\quad\includegraphics{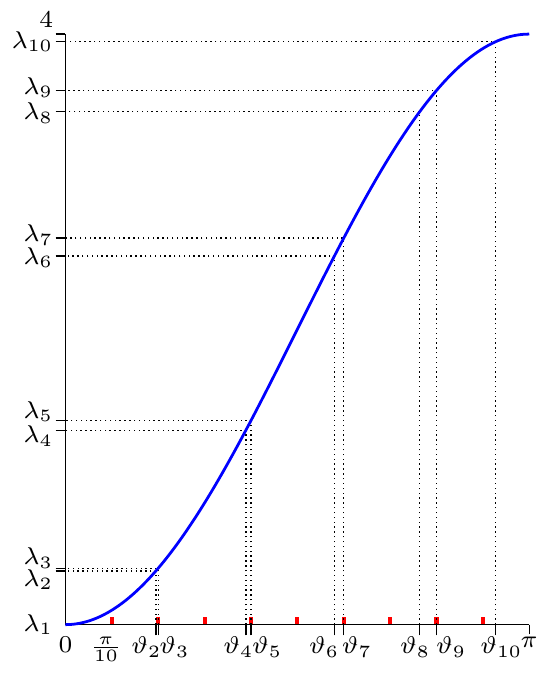}
\caption{The values $\tht_{\al,n,j}$ and $\la_{\al,n,j}$ for $\al=1/3$, $n=10$ (left)
and $\al=4/5$, $n=6$ (right);
the red marks on the horizontal axis correspond to $k\pi/10$, $1\le k\le 9$.
\label{fig:theta_vs_lambda}
}
\end{figure}

\section{Solving the main equation by the fixed-point method}
\label{sec:fixed_point}

We recall that $\eta_{\al,n}$ and $\ka_\al$ are defined by~\eqref{eq:eta} and~\eqref{eq:ka}, respectively, and that $\eta_\al$ does not depend of $n$.
Here are other equivalent formulas for $\eta_\al$: 
\begin{align}
\label{eq:eta1}
\eta_\al(x)
&=\pi-2\arctan\left(\frac{1-\al}{\al}\tan\frac{x}{2}\right),
\\
\label{eq:eta2}
\eta_\al(x)
&=2\arcsin\frac{\ka_\al \cos\frac{x}{2}}{\sqrt{\sin^2\frac{x}{2}+\ka_\al^2\cos^2\frac{x}{2}}},
\\
\label{eq:eta3}
\eta_\al(x)
&=2\arcsin\frac{\sqrt{2}\,\al\cos\frac{x}{2}}%
{\sqrt{(2\al^2-2\al+1)+(2\al-1)\cos(x)}}.
\end{align}
We notice that~\eqref{eq:eta} is more convenient to use if $x$ is close to $\pi$, while~\eqref{eq:eta1} is better for $x$ close to $0$.
The first two derivatives of $\eta_{\al,n}$ are
\begin{align}
\label{eq:etader}
\eta_\al'(x) &= - \frac{\ka_\al(1+\tan^2\frac{x}{2})}{\ka_\al^2+\tan^2\frac{x}{2}},
\\
\label{eq:eta_der_v2}
 \eta_\al'(x) &= - \frac{\ka_\al(1+\cot^2\frac{x}{2})}{1+\ka_\al^2\cot^2\frac{x}{2}},
\\
\label{eq:eta_der_2}
\eta_\al''(x) & =  \frac{(\ka_\al^2-1)\tan\frac{x}{2}}{\ka_\al^2+\tan^2\frac{x}{2}}\,\eta_\al'(x).
\end{align}
The incoming proposition gives some upper bounds for $\eta_\al'$ and $\eta_\al''$ for every $\al$ in $(0,1)$, involving the following numbers:
\begin{equation}\label{eq:K}
    \cK_1(\al) \eqdef \max\left\{\ka_\al,\frac{1}{\ka_\al}\right\},\qquad \cK_2(\al) \eqdef \frac{|\ka_\al^2-1|}{2\ka_\al} \cK_1(\al)=\frac{\cK_1^2(\al)-1}{2}.
\end{equation}

\begin{prop}\label{prop:eta_bound}
Each derivative of $\eta_\al$ is a bounded function on $(0,\pi)$. In particular, 
\begin{align}\label{eq:eta_der_bounded}
    \sup_{0<x<\pi}|\eta_\al'(x)| &= \cK_1(\al), \\ \label{eq:eta_der2_bounded}
    \sup_{0<x<\pi}|\eta_\al''(x)| & \le \cK_2(\al).
\end{align}
\end{prop}

\begin{proof}
In order to prove~\eqref{eq:eta_der_bounded},
we rewrite~\eqref{eq:etader} as follows:
\begin{equation}\label{eq:eta_der_version_2} \eta_\al'(x)=- \ka_\al\left( 1 + \frac{1-\ka_\al^2}{\ka_\al^2+\tan^2\frac{x}{2}} \right) \qquad (x\in (0,\pi)).
\end{equation}
We notice that $\tan^2(x/2)$ increases from $0$ to $\infty$ as $x$ goes from $0$ to $\pi$.
If $0< \al \le1/2$, then $\ka_\al\le 1$,
and $\eta_\al'$ increases taking values from
$\eta_\al'(0) = -\ka_\al^{-1}$ to
$\eta_\al'(\pi) = -\ka_\al$.
If $1/2<\al<1$, then $\eta'_\al$ decreases.
In both cases, the maximal value of $|\eta_\al'|$ is reached at one of the points $0$ or $\pi$.
This proves~\eqref{eq:eta_der_bounded}.

For the second derivative of $\eta_\al''$, from~\eqref{eq:eta_der_2} we get
\[ |\eta_\al''(x)| =   \frac{\tan\frac{x}{2}}{\ka_\al^2+\tan^2\frac{x}{2}} |\ka_\al^2-1| |\eta_\al'(x)| \le \frac{|\ka_\al^2-1|}{2|\ka_\al|}\cK_1(\al) \qquad (x\in (0,\pi)). \]
This is exactly~\eqref{eq:eta_der2_bounded}.

For the higher derivatives of $\eta_{\al,j}$, the explicit estimates are too tedious, and we propose the following argument.
By~\eqref{eq:etader},
$\eta_{\al}'$ is analytic in a neighborhood of $x$,
for any $x$ in $(0,\pi)$. Even more, $\eta_{\al}'$ has an analytic extension in some neighborhoods of the points $0$ and $\pi$.
Hence, $\eta_{\al,j}'$ has an analytic extension to a certain open set in the complex plane containing the segment $[0,\pi]$. Therefore, each derivative of this function is bounded on $(0,\pi)$.
\end{proof}

For every $j$, $1\le j\le n$, we define the function $f_{\al,n,j}\colon[0,\pi]\to\bR$ by
\begin{equation}\label{eq:L_f}
    f_{\al,n,j}(x)\eqdef
    d_{n,j}+\frac{\eta_\al(x)}{n},
\end{equation}
i.e., $f_{\al,n,j}(x)=((j-1)\pi+\eta_\al(x))/n$.
Hence~\eqref{eq:main_eq} can be written as $\tht_{\al,n,j} = f_{\al,n,j}(\tht_{\al,n,j})$.

\begin{prop}\label{prop:Z_contractive_weak}
Let $n> \cK_1(\al)$, and let $j$ be even, $1\le j\le n$. Then $f_{\al,n,j}$ is contractive in $\operatorname{clos}( I_{n,j})$. Its fixed point belongs to $I_{n,j}$ and coincides with $\tht_{\al,n,j}$.
\end{prop}
\begin{proof}
Since $\eta_\al$ takes values in $[0,\pi]$,
for every $x$ in $\operatorname{clos}(I_{n,j})$ we get
\[ \frac{(j-1)\pi}{n}\le\frac{(j-1)\pi+\eta_{\al}(x)}{n}\le \frac{j\pi}{n}, \]
i.e., $f_{\al,n,j}(x)\in\operatorname{clos}(I_{n,j})$. 
By Proposition~\ref{prop:eta_bound}, $\eta_\al'$ is bounded by $\cK_1(\al)$,
hence
\[
\left|f_{\al,n,j}'(x)\right|\le\frac{\cK_1(\al)}{n}<1.
\]
This implies that $f_{\al,n,j}$ is a contractive function on $\operatorname{clos}(I_{n,j})$. 
Then, by the Banach fixed point theorem, $f_{\al,n,j}$ has a unique fixed point, and by Theorem~\ref{thm:weak_characteristic_equation_L} it coincides with $\tht_{\al,n,j}$ and belongs to $I_{n,j}$.
\end{proof}

\begin{cor}
Let $n> \cK_1(\al)$, $j$ be even, $1\le j\le n$, and $x_{\al,n,j}^{(0)}$ be an arbitrary point in $\operatorname{clos}(I_{n,j})$. 
Define the sequence
$\left(x_{\al,n,j}^{(m)}\right)_{m=0}^\infty$ by
\begin{equation*}
    x_{\al,n,j}^{(m)} \eqdef f_{\al,n,j}\left(x_{\al,n,j}^{(m-1)}\right) \qquad (m \ge 1).
\end{equation*}
Then
\begin{equation*}
    \left|x_{\al,n,j}^{(m)} - \tht_{\al,n,j}\right| \le \frac{\pi}{n}\left(\frac{\cK_1(\al)}{n}\right)^m \qquad (m \ge 0).
\end{equation*}
\end{cor}
\begin{proof}
Follows from Proposition~\ref{prop:Z_contractive_weak} and Banach fixed point theorem.
\end{proof}

\section{Newton's method for convex functions}\label{sec:Newton_convex}

In this section we recall some sufficient conditions for the convergence of Newton's method.
Assume that $a,b\in\bR$ with $a<b$; $f$ is differentiable and $f'>0$ on $[a,b]$; there exists $c$ in $[a,b]$ such that $f(c)=0$;
$y^{(0)}$ is a point in $[a,b]$
and the sequence $(y^{(m)})_{m=0}^\infty$ is defined (when possible) by  the recurrence relation
\begin{equation}\label{eq:recursive_newton}
    y^{(m+1)} = y^{(m)} - \frac{f\left(y^{(m)}\right)}{f'\left(y^{(m)}\right)}.
\end{equation}
Obviously, if $y^{(m)}=c$ for some $m$, then the sequence is constant starting from this moment.

In general, the sufficient conditions for Newton's method are quite complicated (see, for example, Kantorovich theorem).
Nevertheless, it is well known that Newton's method converges for convex functions, when the initial point is chosen from the ``correct'' side of the root (\cite[Section 22, Problem 14]{Spivak1994} and \cite[Theorem 2.2]{A1989}).
In the following proposition we show an upper bound for the linear convergence in this case.

\begin{prop}[linear convergence of Newton's method for convex functions]\label{prop:convergence_f_convex_y0>c}
If $f$ is convex on $[a,b]$, $c\le y^{(0)}\le b$, then $y^{(m)}$ belongs to $[c,b]$ for every $m\ge0$, the sequence $(y^{(m)})_{m=0}^\infty$ decreases and converges to $c$, with
\begin{equation}\label{eq:linear_newton}
y^{(m)}-c
\le (b-a)\left(1 - \frac{f'(a)}{f'(b)}\right)^m.
\end{equation}
\end{prop}
\begin{proof}
Reasoning by induction, suppose that $m\ge1$ and $b\ge y^{(m)}\ge c$.
By the mean value theorem, there exists $\xi_m\in[c,y^{(m)}]$ such that $f(y^{(m)})-f(c) = f'(\xi_m)(y^{(m)}-c)$, hence
\begin{equation}\label{eq:mvt_recurrence_newton}
    f(y^{(m)}) = (y^{(m)}-c)f'(\xi_m).
\end{equation}
Combining~\eqref{eq:recursive_newton} with~\eqref{eq:mvt_recurrence_newton} we obtain that
\begin{equation} y^{(m+1)}-c  = y^{(m)}-\frac{f(y^{(m)})}{f'(y^{(m)})}-c = \left(y^{(m)}-c\right) \left(1 - \frac{f'(\xi_m)}{f'(y^{(m)})}\right).  \end{equation}
Since $f'$ is positive and increasing on $[a,b]$,
\begin{equation} 0\le y^{(m+1)}-c \le \left(y^{(m)}-c\right) \left(1 - \frac{f'(a)}{f'(b)}\right). \end{equation}
This yields~\eqref{eq:linear_newton} and the convergence of the sequence.
\end{proof}

The next proposition provides a sufficient convergence condition, when starting from the ``bad'' side of the root. Then $y^{(1)}$ is on the ``good'' side of the root and Proposition~\ref{prop:convergence_f_convex_y0>c} can be applied to the sequence $(y^{(m)})_{m=1}^\infty$.

\begin{prop}\label{prop:first_eval_convex_y0<c}
Suppose that $f$ is convex on $[a,b]$, $a\le y^{(0)} < c$, and
\begin{equation}\label{eq:wrong_side_condition}
    a - \frac{f(a)}{f'(a)} \le b.
\end{equation}
Then $y^{(1)}$ belongs to $[c,b]$.
\end{prop}

\begin{proof}
Since $f$ is convex, its graph is above the tangent lines at the points $(a,f(a))$ and $(y^{(0)},f(y^{(0)}))$.
In particular,
\[
f(y^{(0)})\ge f(a)+f'(a)(y^{(0)}-a),\qquad 0 = f(c) \ge f(y^{(0)})+f'(y^{(0)})(c-y^{(0)}).
\]
Moreover, $f(y^{(0)})<0$ and $f'(y^{(0)})\ge f'(a)>0$.
Hence,
\[
c\le y^{(1)}
=y^{(0)}-\frac{f(y^{(0)})}{f'(y^{(0)})}
\le
y^{(0)}-\frac{f(a)+f'(a)(y^{(0)}-a)}{f'(a)}
=
a-\frac{f(a)}{f'(a)}
\le b.\qedhere
\]
\end{proof}

The following fact is well known~\cite[Theorem 2.1]{A1989}.

\begin{prop}
\label{prop:Newton_convex}
Let $f\in C^2([a,b])$.
Suppose that $(b-a)M<1$, where 
\[ M \eqdef 
\frac{\max_{t\in [a,b]}  |f''(t)|}{2\min_{t\in[a,b]}|f'(t)|}.\]
Assume that  $y^{(m)}$ is well defined and belongs to $[a,b]$ for every $m$.
Then $y^{(m)}$ converges to $c$ as $m$ tends to $\infty$, and for every $m$
\begin{equation}\label{eq:quadratic_convergence}
|y^{(m)} - c| \le  \left((b-a)M\right)^{2^{m}-1} (b-a).
\end{equation}
\end{prop}
\begin{proof}[Idea of the proof]
Let $m\ge0$. By Taylor's formula, there exists $\nu\in[c, y^{(m)}]$ such that \[0=f(c)= f(y^{(m)}) + f'(y^{(m)})\left(c-y^{(m)}\right) + \frac{1}{2}f''(\nu) \left(c-y^{(m)}\right)^2.\]
It follows easily that $ |y^{(m+1)} - c |\le  M (c-y^{(m)})^2$.
Now~\eqref{eq:quadratic_convergence} is obtained by induction. 
\end{proof}

\begin{remark}[Newton's method for concave functions]
\label{rem:Newton_convex_2}
Analogs of Propositions~\ref{prop:convergence_f_convex_y0>c} and~\ref{prop:first_eval_convex_y0<c} hold if $f$ is a concave function.
In this case, each of the following two conditions is sufficient for the convergence:
\begin{itemize}
    \item $a\le y^{(0)}\le c$, 
    \item $c< y^{(0)}\le b$ and $ b-f(b)/f'(b)\ge a$.
\end{itemize} 
Instead of repeating the corresponding proofs with obvious modifications, one can pass to the function $x\mapsto -f(-x)$.
\end{remark}

\section{Solving the main equation by Newton's method}\label{sec:solve_by_Newton}

Recall that $h_{\al,n,j}$ is defined by~\eqref{eq:h_char}.
In this section we prove that the equation $h_{\al,n,j}(x) =0$, which is equivalent to the main equation, can be solved by  Newton's method for every $n\ge3$.

Remark~\ref{rem:case_al=1/2} shows that the eigenvalues can be exactly computed if $\al=1/2$, hence this case could be omitted  in the next propositons.

\begin{prop}[linear convergence of Newton's method applied to the main equation]\label{prop:linear_conv_eigval}
For every $n\ge3$ and every $j$ be even, $1\le j\le n$ and $y_{\al,n,j}^{(0)} \in I_{n,j}$, the sequence $(y_{\al,n,j}^{(m)})_{m=0}^{\infty}$, defined by~\eqref{eq:Newton_sequence_eigval}, converges to
$\tht_{\al,n,j}$. 
The convergence is at least linear:
\begin{equation}\label{eq:sol_linear_newton} \left|y_{\al,n,j}^{(m)}-\tht_{\al,n,j}\right|\le \frac{\pi}{n}\ga_{\al,n}^{m-1},
\end{equation}
where
\begin{equation}
\ga_{\al,n} \eqdef \frac{|2\al-1|}{\al(1-\al)n+|2\al-1|}.
\end{equation}
\end{prop}
\begin{proof}
We start with the case $1/2\le \al \le 1$. 
By the proof of Proposition~\ref{prop:eta_bound},
$\eta_\al$ is analytic in $\clos(I_{n,j}))$,
and $\eta_\al'$ decreases on $[0,\pi]$ taking values
$\eta_\al'(0) = -\ka_\al^{-1}$ to
$\eta_\al'(\pi) = -\ka_\al$.
Therefore, $h_{\al,n,j}$ is analytic and convex on $[0,\pi]$, and
\[
1- \frac{n-\eta_\al'\left(\frac{(j-1)\pi}{n}\right)}{n-\eta_\al'\left(\frac{j\pi}{n}\right)} \le 1 - \frac{n-\eta_\al'(0)}{n-\eta_\al'(\pi)}
= \frac{2\al-1}{\al(1-\al)n + \al^2}\le \frac{2\al-1}{\al(1-\al)n + 2\al-1}.
\]
If $y_{\al,n,j}^{(0)}\ge\tht_{\al,n,j}$,
then Proposition~\ref{prop:convergence_f_convex_y0>c}
yields the convergence and~\eqref{eq:sol_linear_newton}.

For $y_{\al,n,j}^{(0)}<\tht_{\al,n,j}$, we have to verify the condition~\eqref{eq:wrong_side_condition} from Proposition~\ref{prop:first_eval_convex_y0<c}.
In efect, 
\begin{equation*}\label{eq:badside_al_1}
\frac{(j-1)\pi}{n} - \frac{h_{\al,n,j}\left(\frac{(j-1)\pi}{n}\right)}{h_{\al,n,j}'\left(\frac{(j-1)\pi}{n}\right)}
= \frac{(j-1)\pi}{n} + \frac{\eta_{\al,n,j}\left(\frac{(j-1)\pi}{n} \right)}{n-\eta_{\al,n,j}'\left(\frac{(j-1)\pi}{n} \right)}
\le \frac{j\pi}{n}.
\end{equation*}
Since $y_{\al,n,j}^{(1)}\ge\tht_{\al,n,j}$,
after applying $m-1$ steps of the algorithm we get~\eqref{eq:sol_linear_newton}.

For $0\le\al\le1/2$, $h_{\al,n,j}$ is concave,
and the proof of the linear convergence is similar
(see Remark~\ref{rem:Newton_convex_2}).
In particular, if $y_{\al,n,j}^{(0)}>\tht_{\al,n,j}$, then
\begin{equation*}\label{eq:badside_al_2}
\frac{j\pi}{n} - \frac{h_{\al,n,j}\left(\frac{j\pi}{n}\right)}{h_{\al,n,j}'\left(\frac{j\pi}{n}\right)}
= \frac{j\pi}{n} - \frac{\pi-\eta_\al\left(\frac{j\pi}{n}\right)}{n-\eta_\al'\left(\frac{j\pi}{n}\right)}
\ge \frac{(j-1)\pi}{n}.
\end{equation*}
\end{proof}

\begin{proof}[Proof of Theorem~\ref{thm:Newton}]
The first part of Theorem follows from Proposition~\ref{prop:linear_conv_eigval}.
Now we suppose that $0<\al<1$ and $n>\sqrt{\pi\cK_2(\al)/2}$.
Since $\eta_\al'<0$ and $|\eta_\al''|$ is bounded by $\cK_2(\al)$,
\[
M_{\al,n,j}\eqdef \frac{1}{2}\sup_{0<x,y<\pi} \left|\frac{h_{\al,n,j}''(x)}{h_{\al,n,j}'(y)}\right| =\frac{1}{2n} \sup_{0<x,y<\pi}\left| \frac{\eta_{\al}''(x)}{1-\frac{\eta_\al'(y)}{n}} \right|
\le \frac{\cK_2(\al)}{2n}.
\]
Therefore,
$\frac{\pi}{n} M_{\al,n,j}
\le\frac{\pi \cK_2(\al)}{2n^2}<1$, the conditions in Proposition~\ref{prop:Newton_convex} are fulfilled, and we obtain~\eqref{eq:sol_newton_convergence_eigval}.
\end{proof}

The upper bound~\eqref{eq:sol_linear_newton} allows us to compute ``a priori'' the number of steps that will be sufficient to achieve a desired precision.
Namely, if
\begin{equation}\label{eq:bounded_of_steps_for_newton_method}
m >\frac{p + \log_2\frac{\pi}{2n}}{\log_2\frac{1}{\ga_{\al,n}}} + 1,
\end{equation}
then
$|y_{\al,n,j}^{(m)}-\tht_{\al,n,j}|<2^{-p}$.
In fact, after a few iterations, the linear convergence transforms into quadratic convergence, hence reducing the number of iterations.

\section{Asymptotic formulas for the eigenvalues}\label{sec:asymptotic_formulas}

\begin{prop}
\label{prop:theta_approximation_1}
Let $n\ge 3$ and $j$ be even with $1\le j\le n$. Then
\begin{equation}
\label{eq:theta_approximation_1}
\left|\tht_{\al,n,j}-\left(d_{n,j}+\frac{\eta_\al(d_{n,j})}{n}\right)\right|
\le \frac{\pi \cK_1(\al)}{n^2}.
\end{equation}
\end{prop}

\begin{proof}
Theorem~\ref{thm:weak_characteristic_equation_L} assures that
$|\tht_{\al,n,j}-d_{n,j}|\le\frac{\pi}{n}$.
Hence, by the mean value theorem and formula~\eqref{eq:eta_der_bounded},
\[
\bigl|\eta_\al(\tht_{\al,n,j})-\eta_\al(d_{n,j})\bigr|
\le
\|\eta_\al'\|_\infty\,
\left|\tht_{\al,n,j}-d_{n,j}\right|
\le \frac{\pi\cK_1(\al)}{n}.
\]
Using~\eqref{eq:main_eq} we obtain~\eqref{eq:theta_approximation_1}.
\end{proof}

\begin{prop}
\label{prop:weak_theta_asympt}
There exists $C_1(\al)>0$ such that
for every $n\ge3$ and every $j$ even with $1\le j\le n$,
\begin{equation}\label{eq:weak_theta_asympt}
\tht_{\al,n,j}
=d_{n,j}
+\frac{\eta_\al(d_{n,j})}{n}
+\frac{\eta_\al(d_{n,j})\eta_\al'(d_{n,j})}{n^2}
+r_{\al,n,j},
\end{equation}
where $|r_{\al,n,j}|\le\frac{C_1(\al)}{n^3}$.
\end{prop}

\begin{proof}
Proposition~\ref{prop:theta_approximation_1} implies that
\[
\tht_{\al,n,j}
=d_{n,j}+\frac{\eta_\al(d_{n,j})}{n}
+O_\al\left(\frac{1}{n^2}\right).
\]
Substitute this expression into the right-hand side of~\eqref{eq:main_eq}:
\[ 
\tht_{\al,n,j}
=d_{n,j}+\frac{\eta_\al\left(d_{n,j}+\frac{\eta_\al(d_{n,j})}{n}
+O_\al\left(\frac{1}{n^2}\right)\right)}{n}.
\]
Expanding $\eta_\al$
by Taylor's formula around $d_{n,j}$
with two exact term
and estimating the residue term with Proposition~\ref{prop:eta_bound}
we obtain the desired result.
\end{proof}

\begin{proof}[Proof of Theorem~\ref{thm:weak_asympt_weak}]
This theorem follows from Proposition~\ref{prop:weak_theta_asympt}:
we just evaluate $g$ at the expression~\eqref{eq:weak_theta_asympt}
and expand it by Taylor's formula around $d_{n,j}$.
\end{proof}

In a similar manner, iterating in the main equation~\eqref{eq:main_eq}, we could obtain asymptotic expansions with more terms;
see~\cite[(3.9)]{BBGM2018} for the asymptotic expansions up to $n^{-5}$.

There are other forms of the asymptotic expansions for $\la_{\al,n,j}$.
Adding $x$ to both sides of the equation $nx = (j-1)\pi + \eta_\al(x)$ and dividing it over $n+1$, we arrive at the following equivalent form of the main equation:
\begin{equation}\label{eq:main_eq_2}
x =  \frac{j\pi + \widetilde{\eta}_\al(x)}{n+1},\quad\text{where}\quad  \widetilde{\eta}_\al(x) \eqdef \eta_\al(x)+x-\pi
=2\arctan\frac{(\ka_\al-1)\cot\frac{x}{2}}{1+\ka_\al \cot^2\frac{x}{2}}.
\end{equation}
After that, similarly to Proposition~\ref{prop:weak_theta_asympt} and Theorem~\ref{thm:weak_asympt_weak}, we obtain the next result.

\begin{prop}
\label{prop:weak_asympt_2}
There exist $C_2(\al)>0$ and $C_3(\al)>0$ such that
for every $n\ge3$ and every $j$ even with $1\le j\le n$,
\begin{equation}\label{eq:tht_al_den}
    \tht_{\al,n,j} =  \frac{j\pi}{n+1}
+\frac{\widetilde{\eta}_\al\left(\frac{j\pi}{n+1}\right)}{n+1}
+\frac{\widetilde{\eta}_\al\left(\frac{j\pi}{n+1}\right)\widetilde{\eta}_\al'\left(\frac{j\pi}{n+1}\right)}{(n+1)^2} + \widetilde{r}_{\al,n,j},
\end{equation}
\begin{equation}
\label{eq:la_asympt_exp_new}
\begin{aligned}
\la_{\al,n,j}
&=g\left(\frac{j\pi}{n+1}\right)
+\frac{g'\left(\frac{j\pi}{n+1}\right)
\widetilde{\eta}_\al\left(\frac{j\pi}{n+1}\right)}{n+1}
\\
&\qquad+\frac{g'\left(\frac{j\pi}{n+1}\right)
\widetilde{\eta}_\al\left(\frac{j\pi}{n+1}\right)\widetilde{\eta}_\al'\left(\frac{j\pi}{n+1}\right)
+\frac{1}{2}g''\left(\frac{j\pi}{n+1}\right)
\widetilde{\eta}_\al\left(\frac{j\pi}{n+1}\right)^2}{(n+1)^2}
+\widetilde{R}_{\al,n,j},
\end{aligned}
\end{equation}
where $|\widetilde{r}_{\al,n,j}|\le\frac{C_2(\al)}{n^3}$
and $|\widetilde{R}_{\al,n,j}|\le\frac{C_3(\al)}{n^3}$.
\end{prop}

Numerical experiments show that~\eqref{eq:la_asympt_exp_new} is more precise than~\eqref{eq:laasympt_w}, especially for $\al$ close to $1/2$,
but the errors are almost the same for $\al$ close to $1$.
Moreover, $\widetilde{\eta}_\al$ is more complicated than $\eta_\al$
($\widetilde{\eta}_\al$ has two intervals of monotonicity), and the denominator $1/n$ naturally appears in the formula~\eqref{eq:localization_eigvals_weak_odd} for $\la_{\al,n,j}$ with odd $j$.

In the incoming proposition we obtain a simplified asymptotic formula for the eigenvalues $\la_{\al,n,j}$
as $j/n$ tends to zero.

\begin{prop}\label{prop:first_eigenvalues}
Let $\al$ be a fixed number in $(0,1)$.
Then $\la_{\al,n,j}$ has the following asymptotic expansion as $j/n$ tends to $0$:
\begin{equation}
\label{eq:lambda_first_asympt}
    \la_{\al,n,j}
    = \frac{j^2\pi^2}{n^2} - \frac{2j^2(1-\al)\pi^2}{\al n^3} + O_\al\left(\frac{j^4}{n^4}\right).
\end{equation}
\end{prop}

\begin{proof}
First, we use the following Maclaurin's expansions of $\eta_\al$ and $\eta_\al'$:
\[
\eta_\al(x)=\pi-\frac{1-\al}{\al}x+O_\al(x^3),\qquad
\eta_\al'(x)=-\frac{1-\al}{\al}+O_\al(x^2).
\]
Hence, by Proposition~\ref{prop:weak_theta_asympt},
\[
\tht_{\al,n,j}
=\frac{j\pi}{n}-\frac{(1-\al)j\pi}{\al n^2}
+O_\al\left(\frac{j}{n^3}\right)
+O_\al\left(\frac{j^3}{n^4}\right).
\]
We substitute this expansion into 
$g(x)=x^2+O(x^4)$ and obtain~\eqref{eq:lambda_first_asympt}.
\end{proof}

In particular,~\eqref{eq:lambda_first_asympt} can be applied when $j$ is fixed and $n$ tends to $\infty$.
In this situation,~\eqref{eq:lambda_first_asympt} provides a better error estimate than the asymptotic formula in Theorem~\ref{thm:weak_asympt_weak}.

\clearpage

\section{Norms of the eigenvectors}
\label{sec:eigvec_norm}

In this section we prove Theorem~\ref{thm:norm_eigvec} about the eigenvectors of $L_{\al,n,j}$.
We suppose that $\al\in\bC$, $0<\Re(\al)<1$.
Formula~\eqref{eq:eivec_w} follows from Proposition~\ref{prop:eigvec_tri_Toep_corner_per}.
We divide the rest of the proof into three lemmas.
Lemmas~\ref{lem:exact_norm_eigenvector_odd} and~\ref{lem:exact_norm_eigenvector} provide exact formulas~\eqref{eq:exact_norm_eigenvector_odd} and~\eqref{eq:exact_norm_eigenvector_even} for $\|v_{\al,n,j}\|^2$, where $j$ is odd ($j\ge3$) and even, respectively.
In Lemma~\ref{lem:second_term_norm_eigenvector} we prove that for every fixed $\al$ and $j$ even, the second term of~\eqref{eq:exact_norm_eigenvector_even} (which does not contain the factor $n$) is uniformly bounded with respect to $n$ and $j$.

In this section we use the following elementary trigonometric identity:
\begin{equation}\label{eq:sum_cosines}
    \sum_{k=1}^n \cos(2kx+y) = \frac{\sin(nx)\cos((n+1)x+y)}{\sin x}.
\end{equation}

Recall that $v_{\al,n,j}$ is the vector with components~\eqref{eq:eivec_w}.

\begin{lem}
\label{lem:exact_norm_eigenvector_odd}
Let $n\ge 3$ and $j$ be odd, $3\le j\le n$. Then
\begin{equation}
\label{eq:exact_norm_eigenvector_odd}
     \|v_{\al,n,j}\|_2 =   |1-\al|\sqrt{\frac{n}{2} \la_{\al,n,j}}.
\end{equation}
\end{lem}

\begin{proof}
By Theorem~\ref{thm:Localization_weak_eigenvals}, it follows that $\tht_{\al,n,j} = (j-1)\pi/n$ and
\begin{equation*}
    \sin(n\tht_{\al,n,j}) = 0,\qquad  \cos(n\tht_{\al,n,j}) =1,\qquad \sin((n-k)\tht_{\al,n,j}) = -\sin(k\tht_{\al,n,j}).
\end{equation*}
Hence
\begin{align*}
 v_{\al,n,j,k} &=  (1-\overline{\al}) \left(\sin(k\tht_{\al,n,j}) - \sin((k-1)\tht_{\al,n,j})\right) \\&= 2(1-\overline{\al}) \sin\frac{\tht_{\al,n,j}}{2}
    \cos\frac{(2k-1)\tht_{\al,n,j}}{2}.
\end{align*}
Therefore
\begin{equation}
\begin{aligned}
|v_{\al,n,j,k}|^2 &= 4|1-\al|^2\sin^2\frac{\tht_{\al,n,j}}{2}    \cos^2\frac{(2k-1)\tht_{\al,n,j}}{2} \\
& = g(\tht_{\al,n,j})|1-\al|^2 \left(\frac{1+\cos((2k-1)\tht_{\al,n,j})}{2}\right).
\end{aligned}
\end{equation}
Now we sum over $k$ and apply~\eqref{eq:sum_cosines}:
\[
    \|v_{\al,n,j}\|_2^2 = \frac{1}{2}g(\tht_{\al,n,j})|1-\al|^2 \left(n + \frac{\sin(2n \tht_{\al,n,j})}{2\sin \tht_{\al,n,j}} \right).
\]
This implies~\eqref{eq:exact_norm_eigenvector_odd} since $\sin(2n \tht_{\al,n,j}) = \sin(2(j-1)\pi) = 0$.
\end{proof}

For every $x\in[0,\pi]$, we define 
\begin{equation}\label{eq:xi_al}
   \begin{aligned}
      \xi_\al(x) & \eqdef
      \frac{|1-\al|^2}{2} g(x) \cos(\eta_\al(x)) + \frac{|\al|^2}{2}g(\eta_\al(x))\cos(x)\\
      &\pheq+\frac{\Re(\al)-|\al|^2}{2}\left(g(x) + g(x+\eta_\al(x)) -g(\eta_\al(x))\right) - 2|\al|^2\cos(x).
    \end{aligned}
\end{equation}

\begin{lem}
\label{lem:exact_norm_eigenvector}
Let $n\ge 3$ and $j$ be even, $2\le j\le n$. 
Then
\begin{equation}
\label{eq:exact_norm_eigenvector_even}
     \|v_{\al,n,j}\|_2^2 =  n\nu_\al(\tht_{\al,n,j}) + \frac{\sin(\eta_\al(\tht_{\al,n,j}))}{\sin(\tht_{\al,n,j})}\xi_\al(\tht_{\al,n,j}).
\end{equation}
\end{lem}
\begin{proof}
By Theorem~\ref{thm:weak_characteristic_equation_L}, $\tht_{\al,n,j} = (j-1)\pi/n + \eta_\al(\tht_{\al,n,j})/n$.
Then
\[
    \sin(n\tht_{\al,n,j}) = -\sin(\eta_{\al}(\tht_{\al,n,j})),\qquad  \cos(n\tht_{\al,n,j}) =-\cos(\eta_{\al}(\tht_{\al,n,j})),
\]
\[
    \sin((n-k)\tht_{\al,n,j}) = \sin(k\tht_{\al,n,j} - \eta_\al(\tht_{\al,n,j})).
\]
So, \eqref{eq:eivec_w} transforms into
\begin{align*}
v_{\al,n,j,k} &= 2(1-\overline{\al}) \sin\frac{\tht_{\al,n,j}}{2} \cos\frac{(2k-1)\tht_{\al,n,j}}{2} \\& \pheq+ 2\overline{\al} \cos\frac{\eta_\al(\tht_{\al,n,j})}{2} \sin\frac{2k\tht_{\al,n,j} - \eta_\al(\tht_{\al,n,j})}{2}.
\end{align*}
Then $|v_{\al,n,j,k}|^2$ can be written as a sum of three terms:
\begin{equation}\label{eq:v_odd_abs_three_summands}
    \begin{aligned}
        |v_{\al,n,j,k}|^2 & = \frac{|1-\al|^2}{2}g(\tht_{\al,n,j}) (1+\cos((2k-1)\tht_{\al,n,j})) \\& \pheq+ \frac{|\al|^2}{2} \cdot  4\cos^2\frac{\eta_\al(\tht_{\al,n,j})}{2} (1-\cos(2k\tht_{\al,n,j} - \eta_\al(\tht_{\al,n,j}))) \\&\pheq + 4\left(|\al|^2-\Re(\al)\right) \sin\frac{\tht_{\al,n,j}}{2} \cos\frac{\eta_\al(\tht_{\al,n,j})}{2}\times\\&\pheq \times \left(\sin\frac{\eta_\al(\tht_{\al,n,j})-\tht_{\al,n,j}}{2} - \sin\left(2k\tht_{\al,n,j} - \frac{\eta_\al(\tht_{\al,n,j})+\tht_{\al,n,j}}{2}\right)\right).
    \end{aligned}
\end{equation}
Now we compute $\sum_{k=1}^n |v_{\al,n,j,k}|^2$ working separately with each of the three terms from~\eqref{eq:v_odd_abs_three_summands}. 
The sums involving $k\tht_{\al,n,j}$ are transformed by~\eqref{eq:sum_cosines}:
\begin{equation}
    \sum_{k=1}^{n}\cos((2k-1)\tht_{\al,n,j})= \frac{\sin(\eta_\al(\tht_{\al,n,j}))\cos(\eta_\al(\tht_{\al,n,j}))}{\sin(\tht_{\al,n,j})},
\end{equation}
\begin{equation}
        \sum_{k=1}^n\cos(2k\tht_{\al,n,j} - \eta_\al(\tht_{\al,n,j}))=  \frac{\sin(\eta_\al(\tht_{\al,n,j})) \cos(\tht_{\al,n,j}) }{\sin(\tht_{\al,n,j})},
    \end{equation}
\begin{equation}
\label{eq:sum3}
    \sum_{k=1}^n \sin\left(2k\tht_{\al,n,j} - \frac{\eta_\al(\tht_{\al,n,j})+\tht_{\al,n,j}}{2}\right)  = \frac{\sin(\eta_\al(\tht_{\al,n,j})) \sin\frac{\eta_\al(\tht_{\al,n,j}) + \tht_{\al,n,j}}{2}}{\sin(\tht_{\al,n,j})}.
\end{equation}
After some elementary simplifications we obtain~\eqref{eq:exact_norm_eigenvector_even}.
\end{proof}

In the next lemma we prove that the second term in~\eqref{eq:exact_norm_eigenvector_even} is uniformly bounded with respect to $n$ and $j$. 

\begin{lem}
\label{lem:second_term_norm_eigenvector}
There exists $C_{\al}>0$, depending only on $\al$, such that for every $n\ge3$ and every $j$ even, $2\le j\le n$,
\begin{equation}\label{eq:bound_second_eigvec_term}
    \left|\frac{\sin(\eta_\al(\tht_{\al,n,j}))}{\sin(\tht_{\al,n,j})}\xi_\al(\tht_{\al,n,j})\right|\le C_{\al}.
\end{equation}
\end{lem}
\begin{proof}
Obviously, $\xi_\al$ is a bounded function on $[0,\pi]$.
By a simple application of l'H\^opital's rule, the quotient $\sin(\eta_{\al}(x))/\sin(x)$ has finite limits at $0$ and $\pi$, hence it is bounded on $[0,\pi]$.
This implies~\eqref{eq:bound_second_eigvec_term}.
\end{proof}

\begin{proof}[Proof of Theorem~\ref{thm:norm_eigvec}]
It is a well-known basic fact in the theory of laplacian matrices that the vector $[1,\ldots,1]^\top$ is an eigenvector associated to the eigenvalue $\la=0$.
From Proposition~\ref{prop:eigvec_tri_Toep_corner_per} we obtain~\eqref{eq:eivec_w}. 
In Lemma~\ref{lem:exact_norm_eigenvector_odd},~\eqref{eq:norm_eigvec_j_odd} has been proved.
From Lemmas~\ref{lem:exact_norm_eigenvector} and~\ref{lem:second_term_norm_eigenvector} we obtain~\eqref{eq:norm_eigvec_j_even}. 
\end{proof}

\section{Numerical experiments}\label{sec:num_exp}

With the help of Sagemath,
we have verified numerically (for many values of parameters) the representations~\eqref{eq:L_char_pol_via_Cheb},~\eqref{eq:char_pol_fact},~\eqref{eq:charpol_factorization_trig} for the characteristic polynomial,
the equivalence of the formulas~\eqref{eq:eta}, \eqref{eq:eta1}, \eqref{eq:eta2}, \eqref{eq:eta3}
for $\eta_\al$,
expressions~\eqref{eq:exact_norm_eigenvector_odd}, \eqref{eq:exact_norm_eigenvector_even} for the norms of the eigenvectors, 
and some other exact formulas of this paper.

The following web page (written in JavaScript and SVG) contains interactive analogs of Figures~\ref{fig:nx_plus_jpi_eq_eta} and \ref{fig:theta_vs_lambda}, where the user can choose the values of $\al$ and $n$.
\begin{center}
\myurl{https://www.egormaximenko.com/plots/laplacian\_of\_cycle\_eig.html}
\end{center}

We introduce the following notation for different approximations of the eigenvalues.
\begin{itemize}

\item $\la_{\al,n,j}^{\text{gen}}$ are the eigenvalues computed in Sagemath by general algorithms, with double-precision arithmetic.
\item $\la_{\al,n,j}^{\text{N}}
\eqdef g(\tht_{\al,n,j}^{\text{N}})$,
where $\tht_{\al,n,j}^{\text{N}}$ is the numerical solution of the equation $h_{\al,n,j}(x) = 0$ by Newton's method, see Theorem~\ref{thm:Newton}. 
We use $d_{n,j}$ as the initial approximation.
These computations are performed in the high-precision  arithmetic with $3322$ binary digits
($\approx 1000$ decimal digits).

\item Using $\tht_{\al,n,j}^{\text{N}}$ 
we compute $v_{\al,n,j}$ by~\eqref{eq:eivec_w}.

\item $\la_{\al,n,j}^{\text{bisec}}$
is similar to $\la_{\al,n,j}^{\text{N}}$,
but now we solve the equation $h_{\al,n,j}(x) = 0$ by the bisection method, see Proposition~\ref{prop:h_change_sign}.

\item $\la_{\al,n,j}^{\text{fp}}$
is computed similarly to $\la_{\al,n,j}^{\text{N}}$,
but solving the main equation by the fixed point iteration,
see Proposition~\ref{prop:Z_contractive_weak}.

\item $\la_{\al,n,j}^{\text{N},2}$ is computed similarly to
$\la_{\al,n,j}^{\text{N}}$, but using only two iterations of Newton's method. 

\item $\la_{\al,n,j}^{\text{asympt}}$
is the approximation given by~\eqref{eq:laasympt_w}.
\end{itemize}

We have constructed a large series of examples including all rational values $\al$ in $(0,1)$ with denominators $\le10$
and all $n$ with $3\le n\le 256$.
In all these examples, we have obtained
\[
\max_{1\le j\le n}
\|L_{\al,n}v_{\al,n,j}-\la_{\al,n,j}^{\text{N}}v_{\al,n,j}\|_2
<10^{-996},\qquad
\max_{1\le j\le n}|\la_{\al,n,j}^{\text{gen}}-\la_{\al,n,j}^{\text{N}}|
<10^{-13}.
\]
Moreover, in all examples
\[
\max_{1\le j\le n}
|\la_{\al,n,j}^{\text{N}}-\la_{\al,n,j}^{\text{bisec}}|
<10^{-998},
\]
and for $n>\cK_1(\al)$,
\[
\max_{1\le j\le n}
|\la_{\al,n,j}^{\text{fp}}-\la_{\al,n,j}^{\text{N}}|
<10^{-998}.
\]
For Theorem~\ref{thm:weak_asympt_weak},
we have computed the errors
\[
R_{\al,n,j}^{\text{asympt}}\eqdef \la_{\al,n,j}^{\text{asympt}}-\la_{\al,n,j}^{\mathrm{N}}
\]
and their maximums
$\|R_{\al,n}^{\text{asympt}}\|_\infty=\max_{1\le j\le n}|R_{\al,n,j}^{\text{asympt}}|$.
Table~\ref{table:errors_weak_asympt} shows that these errors indeed can be bounded by $O_\al(1/n^3)$.

\begin{table}[htb]
\caption{Values of $\|R_{\al,n}^{\text{asympt}}\|_\infty$
and $n^3 \|R_{\al,n}^{\text{asympt}}\|_\infty$
for some $\al$ and $n$.
\label{table:errors_weak_asympt}}
\[
\begin{array}{|c|c|c|}
\hline
\multicolumn{3}{ |c| }{\bigstrut\al=1/3}
\\\hline
\bigstrut n & \|R_{\al,n}^{\text{asympt}}\|_\infty &%
\hstrut{}n^3 \|R_{\al,n}^{\text{asympt}}\|_\infty\hstrut{} 
\\\hline
\medstrut 256 & 2.28\times10^{-6} & 38.24
\\
\medstrut 512 & 2.90\times10^{-7} & 38.86
\\
\medstrut 1024 & 3.65\times10^{-8} & 39.17 
\\
\medstrut 2048 & 4.58\times10^{-9} & 39.32 
\\
\medstrut\hstrut{}4096\hstrut{}&%
\hstrut{}5.73\times10^{-10}\hstrut{}&%
\hstrut{}39.40\hstrut{}
\\
\medstrut\hstrut{}8192\hstrut{}&%
\hstrut{}7.17\times10^{-11}\hstrut{}&%
\hstrut{}39.44\hstrut{}
\\
\hline
\end{array}
\qquad
\begin{array}{|c|c|c|}
\hline
\multicolumn{3}{ |c| }{\bigstrut\al=4/5}
\\\hline
\bigstrut n & \|R_{\al,n}^{\text{asympt}}\|_\infty &%
\hstrut{}n^3 \|R_{\al,n}^{\text{asympt}}\|_\infty\hstrut{}%
\\\hline
\medstrut 256 & 6.90\times10^{-7} & 11.58 \\
\medstrut 512 & 8.66\times10^{-8} & 11.62 \\
\medstrut 1024 & 1.08\times10^{-8} & 11.63 \\
\medstrut 2048 & 1.36\times10^{-9} & 11.64 \\
\medstrut\hstrut{}4096\hstrut{}&%
\hstrut{}1.69\times10^{-10}\hstrut{}&%
\hstrut{}11.64\hstrut{}\\
\medstrut\hstrut{}8192\hstrut{}&%
\hstrut{}2.12\times10^{-11}\hstrut{}&%
\hstrut{}11.64\hstrut{}\\
\hline
\end{array}
\]
\end{table}

Let $R_{\al,n,j}^{\text{N},2}\eqdef \la_{\al,n,j}^{\text{N},2}-\la_{\al,n,j}^{\mathrm{N}}$
and
$\|R_{\al,n}^{\text{N},2}\|_\infty=\max_{1\le j\le n}|R_{\al,n,j}^{\text{N},2}|$.
Table~\ref{table:errors_newton_2_iter} shows that these errors behave indeed as $O_\al(1/n^7)$.

\begin{table}[htb]
\caption{Values of $\|R_{\al,n}^{\text{N},2}\|_\infty$
and $n^7 \|R_{\al,n}^{\text{N},2}\|_\infty$
for some $\al$ and $n$.
\label{table:errors_newton_2_iter}}
\[
\begin{array}{|c|c|c|}
\hline
\multicolumn{3}{ |c| }{\bigstrut\al=1/3}
\\\hline
\bigstrut n & \|R_{\al,n}^{\text{N},2}\|_\infty &%
\hstrut{}n^7 \|R_{\al,n}^{\text{N},2}\|_\infty\hstrut{}%
\\\hline
\medstrut 256 & 4.13\times10^{-17} & 2.97 \\
\medstrut 512 & 3.26\times10^{-19} & 3.01 \\
\medstrut 1024 & 2.57\times10^{-21} & 3.03 \\
\medstrut 2048 & 2.01\times10^{-23} & 3.04 \\
\medstrut\hstrut{}4096\hstrut{}&%
\hstrut{}1.57\times10^{-25}\hstrut{}&%
\hstrut{}3.04\hstrut{}\\
\medstrut\hstrut{}8192\hstrut{}&%
\hstrut{}1.23\times10^{-27}\hstrut{}&%
\hstrut{}3.05\hstrut{}\\
\hline
\end{array}
\qquad
\begin{array}{|c|c|c|}
\hline
\multicolumn{3}{ |c| }{\bigstrut\al=4/5}
\\\hline
\bigstrut n & \|R_{\al,n}^{\text{N},2}\|_\infty &%
\hstrut{}n^7 \|R_{\al,n}^{\text{N},2}\|_\infty\hstrut{}%
\\\hline
\medstrut 256 & 6.30\times10^{-16} & 45.41\\
\medstrut 512 & 5.02\times10^{-18} & 46.33\\
\medstrut 1024 & 3.96\times10^{-20} & 46.80\\
\medstrut 2048 & 3.11\times10^{-22} & 47.04\\
\medstrut\hstrut{}4096\hstrut{}&%
\hstrut{}2.44\times10^{-24}\hstrut{}&%
\hstrut{}47.16\hstrut{}\\
\medstrut\hstrut{}8192\hstrut{}&%
\hstrut{}1.91\times10^{-26}\hstrut{}&%
\hstrut{}47.22\hstrut{}\\
\hline
\end{array}
\]
\end{table}

We have done similar tests for many other values of $\al$ and $n$. 
Numerical experiments show that $n^3\|R^{\text{asympt}}_{\al,n}\|_\infty$ and $n^7\|R^{\text{N}}_{\al,n}\|_\infty$
are bounded by some numbers depending on $\al$, and that numbers grow as $\al$ tends to $0$ or $1$. 

Let $R_{\al,n,j}^{\text{asympt},2}\eqdef \la_{\al,n,j}^{\text{N}}- (\frac{j^2\pi^2}{n^2} - \frac{2(1-\al)j^2\pi^2}{\al n^3})$.
Table~\ref{table:errors_asympt_first_eigenvalues} shows that these errors behave indeed as $O_\al(j^4/n^4)$.

\begin{table}[htb]
\caption{Values of $(n^4/j^4)|R_{\al,n,j}^{\text{asympt},2}|$ for $\al=1/3$, and some $n$ and even $j$.
\label{table:errors_asympt_first_eigenvalues}}
\[
\begin{array}{|c|c|c|c|}
\hline
\multicolumn{4}{ |c| }{\bigstrut\al=1/3}
\\\hline
\bigstrut n & (n^4/2^4) |R_{\al,n,2}^{\text{asympt},2}| &%
\hstrut{} (n^4/4^4) |R_{\al,n,4}^{\text{asympt},2}|\hstrut{} & (n^4/6^4) |R_{\al,n,6}^{\text{asympt},2}|
\\\hline
\medstrut 256 &
21.80 & 0.18 & 4.25 \\
\medstrut 512 &
21.65 & 0.44 & 4.53 \\
\medstrut 1024 &
21.57 & 0.58 & 4.67 \\
\medstrut 2048 &
21.53 & 0.65 & 4.75 \\
\medstrut\hstrut{}4096\hstrut{}&
\hstrut{}21.51  \hstrut{}&%
\hstrut{}0.68 \hstrut{}&
4.79\\
\medstrut\hstrut{}8192\hstrut{}&
\hstrut{}21.50 \hstrut{}&%
\hstrut{}0.70 \hstrut{}&
4.81\\
\hline
\end{array}
\]
\end{table}

\begin{thebibliography}{20}

\bibitem{A1989}
Atkinson, K. E.: An Introduction to Numerical Analysis. 2nd ed. Wiley, New York (1989).

\bibitem{BBGM2018}
Barrera, M.; B\"{o}ttcher, A.,
Grudsky, S. M.; Maximenko, E. A.:
Eigenvalues of even very nice Toeplitz matrices can be unexpectedly erratic.
In: 
B\"{o}ttcher, A., Potts, D., Stollmann, P., Wenzel, D. (eds.) The Diversity and Beauty of Applied Operator Theory, 51--77.
Operator Theory: Advances and Applications, vol. 268.
Birkh\"{a}user, Cham (2018),
\doi{10.1007/978-3-319-75996-8\_2}.

\bibitem{BPZ2020}
Basak, A.; Paquette, E.; Zeitouni, O.: Spectrum of random perturbations of Toeplitz matrices with finite symbols. Trans. Amer. Math. Soc. 373, 4999--5023 (2020),
\doi{10.1090/tran/8040}.

\bibitem{BGM2017}
Bogoya, J. M.; Grudsky, S. M.; Maximenko, E. A.:
Eigenvalues of Hermitian Toeplitz matrices generated by simple-loop symbols with relaxed smoothness.
In: Bini, D.; Ehrhardt, T.;
Karlovich, A.; Spitkovsky, I. (eds.)
Large Truncated Toeplitz Matrices,
Toeplitz Operators, and Related Topics, 179--212.
Operator Theory: Advances and Applications,
vol. 259.
Birkh\"{a}user, Cham (2017), \doi{10.1007/978-3-319-49182-0\_11}.

\bibitem{BBGM2015}
Bogoya, J. M.; B\"{o}ttcher, A.;
Grudsky, S. M.; Maximenko, E. A.:
Eigenvalues of Hermitian Toeplitz matrices with smooth simple-loop symbols.
J. Math. Anal. Appl. 422, 1308--1334 (2015),
\doi{10.1016/j.jmaa.2014.09.057}.

\bibitem{BBGM2017}
B\"{o}ttcher, A.; Bogoya, J. M.;
Grudsky, S. M.; Maximenko, E. A.:
Asymptotic formulas for the eigenvalues and eigenvectors of Toeplitz matrices.
Sb. Math. 208, 1578--1601 (2017),
\doi{10.1070/SM8865}.


\bibitem{BFGM2014}
Böttcher, A.; Fukshansky, L.; Garcia, S. R.; Maharak, H.: Toeplitz determinants with perturbations in the corners. J. Funct. Anal. 268, 171--193 (2014),
\doi{10.1016/j.jfa.2014.10.023}.

\bibitem{BYR2006}
Britanak, V.; Yip, P. C.; Rao, K. R.:
Discrete Cosine and Sine Transforms:
General Properties, Fast Algorithms and Integer Approximations.
Academic Press, San Diego (2006).

\bibitem{Ferguson1980}
Ferguson, W. E.:
The construction of Jacobi
and periodic Jacobi matrices
with prescribed spectra.
Math. Comput. 35:152,
1203--1220 (1980),
\doi{10.2307/2006386}.

\bibitem{FF2009}
Fernandes, R.; da Fonseca, C. M.:
The inverse eigenvalue problem for Hermitian matrices whose graphs are cycles.
Linear Multilinear Alg., 57, 673--682 (2009), \doi{10.1080/03081080802187870}.

\bibitem{FK2020}
Da Fonseca, C. M.; Kowalenko, V.: 
Eigenpairs of a family of tridiagonal matrices: three decades later.
Acta Math. Hung. 160, 376--389 (2020), \doi{10.1007/s10474-019-00970-1}.

\bibitem{DV2009}
Da Fonseca, C. M.; Veerman, J. J. P.:
On the spectra of certain directed paths.
Appl. Math. Lett. 22, 1351--1355 (2009),
\doi{10.1016/j.aml.2009.03.006}.


\bibitem{GS2017}
Garoni, C.; Sierra-Capizzano, S.:
Generalized Locally Toeplitz Sequences:
Theory and Applications. Volume I.
Springer, Cham (2017).

\bibitem{GT2009}
Grassmann, W. K.; Tavakoli, J.: 
Spectrum of certain tridiagonal matrices when their dimension goes to infinity. 
Linear Algebra Appl. 431, 1208–1217 (2009),
\doi{10.1016/j.laa.2009.04.013}.

\bibitem{GMS2021}
Grudsky, S. M.; Maximenko, E. A.; Soto-Gonz\'alez, A.:
Eigenvalues of tridiagonal Hermitian Toeplitz matrices with perturbations in the off-diagonal corners. In: Karapetyants, A. N.; Kravchenko, V. V.; Liflyand, E.; Malonek, H. R. (eds.)
Operator Theory and Harmonic Analysis. OTHA 2020. Springer Proceedings in Mathematics \& Statistics, vol 357. Springer, Cham (2021),
\doi{10.1007/978-3-030-77493-6\_11}.

\bibitem{HornJohnson2013}
Horn, R. A.; Johnson, C. R.:
Matrix Analysis. 2nd ed.
Cambridge University Press, New York (2013).

\bibitem{KST1999}
Kulkarni, D.; Schmidt, D.; Tsui, S.: 
Eigenvalues of tridiagonal pseudo-Toeplitz matrices.
Linear Algebra Appl.
297, 63--80 (1999),
\doi{10.1016/S0024-3795(99)00114-7}.

\bibitem{LWHF2014}
Lin, Z.; Wang, L.; Han, Z.; Fu, M.: Distributed formation control of multi-agent
systems using complex Laplacian. IEEE Transactions on automatic control, 59, 1765--1777 (2014), \doi{10.1109/TAC.2014.2309031}.

\bibitem{M2012}
Molitierno, J. J.:
Applications of Combinatorial Matrix Theory to Laplacian Matrices of Graphs. CRC Press, Florida (2012).

\bibitem{NR2019}
Noschese, S.; Reichel, L.:
Eigenvector sensitivity under general and structured perturbations of tridiagonal Toeplitz-type matrices.
Numer. Linear Algebra Appl.
26, e2232 (2019),
\doi{10.1002/nla.2232}.

\bibitem{OA2014}
\"Otele\c{s}, A.; Akbulak, M.:
Positive integer powers of one type of complex tridiagonal matrix. 
Bull. Malays. Math. Sci. Soc. (2) 37, 971--981 (2014),
\myurl{http://math.usm.my/bulletin/pdf/v37n3/v37n4p6.pdf}.

\bibitem{R2017}
Reyes-Lega, A. F.:
Some aspects of operator algebras in quantum physics. In: Cano, L.; Arboleda, S.; Cardona, A.; Ocampo, H.; Reyes-Lega, A. F. (eds.)
Geometric, Algebraic and Topological Methods for Quantum Field Theory. World Scientific, 1--74 (2016),
\doi{10.1142/9789814730884\_0001}. 

\bibitem{SM2014}
Strang, G.; MacNamara, S.:
Functions of difference matrices are Toeplitz plus Hankel.
SIAM Rev. 56, 525--546 (2014),
\doi{10.1137/120897572}.

\bibitem{Spivak1994}
Spivak, M.: Calculus, 3rd ed. Publish or Perish, Houston (1994).

\bibitem{Tilli1998}
Tilli, P.:
Locally Toeplitz sequences: spectral properties and applications.
Linear Algebra Appl. 278, 91--120 (1998),
\doi{10.1016/S0024-3795(97)10079-9}.

\bibitem{Tyrtyshnikov1996}
Tyrtyshnikov, E. E.:
A unifying approach to some old and new theorems on distribution and clustering.
Linear Algebra Appl. 232, 1--43 (1996),
\doi{10.1016/0024-3795(94)00025-5}.

\bibitem{TS2017}
Tavakolipour, H.; Shakeri, F.: On tropical eigenvalues of tridiagonal Toeplitz matrices. Linear Algebra Appl. 539, 198--218 (2017),
\doi{10.1016/j.laa.2017.11.009}.

\bibitem{VHB2018}
Veerman, J. J. P.; Hammond, D. K.; Baldivieso, P. E.:
Spectra of certain large tridiagonal matrices.
Linear Algebra Appl. 548, 123–147 (2018),
\doi{10.1016/j.laa.2018.03.005}.

\bibitem{W2008}
Willms, A. R.: Analytic results for the eigenvalues of certain tridiagonal matrices. Siam J. Matrix Anal. Appl. 30, 639--656 (2008),
\doi{10.1137/070695411}.

\bibitem{YuehCheng2008}
Yueh, W. C.; Cheng, S. S.:
Explicit eigenvalues and inverses of tridiagonal Toeplitz matrices with four perturbed corners.
ANZIAM J. 49, 361--387 (2008),
\doi{10.1017/S1446181108000102}.

\bibitem{Z2014}
Zampieri, G.:
Involutions of real intervals.
Annales Polonici Mathematici 112, 25--35 (2014),
\doi{10.4064/ap112-1-2}.

\bibitem{ZJJ2019}
Zhang, M.; Jiang, X.; Jiang, Z.: Explicit determinants, inverses andeigenvalues of four band Toeplitz matrices with perturbed rows. Special Matrices 7,  52--66 (2019),
\doi{10.1515/spma-2019-0004}.

\end{thebibliography}
\end{document}